\documentclass[12,reqno]{amsart}

\usepackage{amsfonts}
\usepackage[active]{srcltx}
\usepackage{amsmath}
\usepackage{amssymb}
\usepackage{amsthm}
\usepackage{mathtools}
\usepackage{bm} 
\usepackage{empheq}

\usepackage{enumitem}
\usepackage{wasysym}
\usepackage{verbatim}
\usepackage{graphicx}
\usepackage[
bookmarks=true,         
bookmarksnumbered=true, 
colorlinks=true, pdfstartview=FitV, linkcolor=blue, citecolor=blue,
urlcolor=blue]{hyperref}
\usepackage{microtype}
\usepackage{cancel}
\usepackage[normalem]{ulem}
\usepackage{soul}

\theoremstyle{plain}\newtheorem{theorem}{Theorem}[section]

\newtheorem{corollary}[theorem]{Corollary}

\newtheorem{lemma}[theorem]{Lemma}
\newtheorem{problem}[theorem]{Problem}
\newtheorem{proposition}[theorem]{Proposition}

\renewenvironment{proof}[1][Proof]{\textbf{#1.} }{\ \rule{0.5em}{0.5em} \par }
\theoremstyle{remark}

\theoremstyle{definition}
\newtheorem{remark}[theorem]{Remark}
\newtheorem{definition}[theorem]{Definition}

\newtheorem{assumption}[theorem]{Assumption}

\def\PP{\mathbb{P}}
\def\RR{\mathbb{R}}
\def\EE{\mathbb{E}}

\def\cN{\mathcal{N}}
\def\Tr{{\rm Tr}}

\def\cB{{\mathcal B}}

\def\cC{{\mathcal C}}
\def\cE{{\mathcal E}}
\def\cF{{\mathcal F}}

\def\cH{{\mathcal H}}

\def\cP{{\mathcal P}}

\def\si{{\sigma}}

\def\cP{{\mathcal  P}}
\def\De{{\Delta}}
\def\et{{\eta}}

\def\Om{{\Omega}}

\def\al{{\alpha}}

\def\ga{{\gamma}}
\def\De{{\Delta}}

\def\si{{\sigma}}

\def\tr{{ \hbox{ Tr} }}

\def\La{{\Lambda}}
\def\vare{{\varepsilon}}

\def\si{{\sigma}}
\def\al{{\alpha}}

\def\Dom{\hbox{Dom}}

\newcommand{\D}{{\mathbb D}}

\renewcommand{\theequation}{\arabic{section}.\arabic{equation}}

\newcommand{\ep}{\ensuremath{\varepsilon}}

\newcommand{\ce}{\begin{eqnarray*}}
	\newcommand{\de}{\end{eqnarray*}}
\newcommand{\cen}{\begin{eqnarray}}
	\newcommand{\den}{\end{eqnarray}}

\def\mR{\mathbb{R}}
\def\mI{\mathbb{I}}

\def\mC{{\mathbb C}}
\def\mD{{\mathbb D}}
\def\mE{{\mathbb E}}

\def\mI{{\mathbb I}}

\def\mP{{\mathbb P}}

\def\mR{{\mathbb R}}

\def\red{\color{red}}

\def\diam{\mathrm{diam}}
\def\dist{\hbox{Dist}}

\def\Om{\Omega}

\def\<{{\langle}}
\def\>{{\rangle}}
\def\({{\Big(}}
\def\){{\Big)}}

\def\bx{{\mathbf{x}}}
\def\tr{{\rm tr}}

\def\min{{\mathord{{\rm min}}}}

\def\bt{\begin{theorem}}
	\def\et{\end{theorem}}
\def\bl{\begin{lemma}}
	\def\el{\end{lemma}}
\def\br{\begin{remark}}
	\def\er{\end{remark}}
\def\bx{\begin{Example}}
	\def\ex{\end{Example}}
\def\bd{\begin{definition}}
	\def\ed{\end{definition}}
\def\bp{\begin{proposition}}
	\def\ep{\end{proposition}}
\def\bc{\begin{corollary}}
	\def\ec{\end{corollary}}
\def\ba{\begin{eqnarray}}
	\def\ea{\end{eqnarray}}
\def\bas{\begin{eqnarray*}}
	\def\eas{\end{eqnarray*}}

\def\cB{{\mathcal B}}
\def\cC{{\mathcal C}}

\def\cE{{\mathcal E}}
\def\cF{{\mathcal F}}

\def\cH{{\mathcal H}}

\def\cN{{\mathcal N}}

\def\cP{{\mathcal P}}

\def\cS{{\mathcal S}}

\renewcommand{\theequation}{\arabic{section}.\arabic{equation}}
\let\Section=\section
\def\section{\setcounter{equation}{0}\Section}

\begin{document}
	
	\title[Feynman-Kac formula]{Feynman-Kac formula for general time dependent stochastic parabolic equation on a bounded domain and applications}
	
	\author{ \sc Yaozhong Hu  }
	\thanks{Y.Hu  is supported by an  NSERC discovery fund RGPIN
2024-05941 and a centennial  fund of University of Alberta,
		National Natural Science Foundation of China (12261046).}
	\address{Department of Mathematical and  Statistical
		Sciences, University of Alberta at Edmonton,
		Edmonton, Canada, T6G 2G1}
	\email{ yaozhong@ualberta.ca
	}
	\author{\sc Qun Shi }
	\address{School of Mathematics and Statistics,
		Jiangxi Normal University, \\
		Nanchang, Jiangxi 330022, P.R.China}
	\email{shiq3@mail2.sysu.edu.cn}
	\thanks{Q.Shi is supported by China Overseas Education Fund Committee,
		National Natural Science Foundation of China (11901257, 12261046), Jiangxi Provincial Natural Science foundation (20242BAB20001).\\
		{\bf Keywords}: Exponential integrability, fractional Brownian sheet, Malliavin Calculus, Eigenvalue.
	}
	\thanks{
			AMS Subject class[2010]: 60H15, 35K15, 35R60   60G15, 60G22, 60J35,    60H07; 60H10; 65C30, 47D07
	}
	\date{}
	\maketitle
	%
	%
	%

	\begin{abstract}
		This paper  establishes  a  Feynman-Kac formula to represent the solution to  general  time inhomogeneous   stochastic parabolic partial differential equations driven by  multiplicative fractional Gaussian noises    in   bounded domain: $\frac{\partial u(t,x)}{\partial t}=L_tu(t,x)+u(t,x)\dot{W}(t,x)$,  where $L_t$ is a second order uniformly elliptic operator whose coefficients  can depend on time and generates a time inhomoegenous Markov process.    The
		idea is to use   the  Aronson  bounds  of fundamental solution of the associated heat   kernel and   the techniques   from Malliavin calculus.
		The newly obtained   Feynman-Kac formula is then applied to  establish the  H\"older regularity in the space and time variables.  The   dependence on time of the      coefficients poses serious challenges and new results about the stochastic differential equations are discovered to face the challenge.  An amazing application of the Feynman-Kac formula is about  the matching upper and
		lower bounds for all    moments of the solution,
		an critical tool for the intermittency. For the latter result we need first to establish
		new small ball like bounds   for  the diffusion associated with the parabolic differential operator
		in the bounded domain which is of interest on its own.
	\end{abstract}


\setcounter{equation}{0}
\renewcommand{\theequation}{\arabic{section}.\arabic{equation}}
\numberwithin{equation}{section}

\maketitle

\section{\bf Introduction}
Feynman-Kac formula  originates from   Feynman integral  introduced    by
R. Feynman \cite{feynman} in an  attempt to give a mathematical  description of the
quantum field theory and also  to  give  a representation of the solution of   Schr\"odinger's
equation.  Feynman integral  is    one of the most important
tools in quantum physics. Feynman-Kac formula is the Feynman integral representation of the solution  of the Schr\"odinger equation  when time is complexified, giving a probabilistic representation of the solution of heat equation via Brownian motion.
Now Feynman-Kac formula has been extended to general Markov processes
with general deterministic potentials,    building  a bridge between stochastic analysis and partial differential equations (PDEs) by providing a probabilistic  representation  for the  solution  of linear   PDEs.  It can be used to study the deep properties of the solution and so on.
  The fact  that   solutions of certain linear   PDEs can be expressed by an appropriate average  of It\^o processes associated with these PDEs    opens,  for example,   the door for derivation of sampling-based Monte Carlo approximation methods, which can be mesh-free and thereby stands a chance to approximate the  solutions of PDEs without suffering from curse of dimensionality.

The extension of the  Feynman-Kac formula to stochastic heat equations  has also been studied.     Hu et al \cite{Hu} obtain  the following stochastic  Feynman-Kac formula
\begin{equation}\label{sec1-feynman}
	\begin{split}
		u(t,x)
		=&\mE^B\left[f(B^x_t)\exp \Big(\int_0^tW(dr,B^x_t-r)\Big)\right]\\
		=&\mE^B\left[f(B^x_t)\exp\Big(\int_0^t\int_{\mR^d}\delta (B^x_t-r-y)W(dr,dy)\Big)\right]
	\end{split}
\end{equation}
for the solution to the heat equation   with fractional noise
on the whole space $\mR^d$:
\begin{equation}\label{sec1-eq5}
	\left\{
	\begin{array}{lll}
		\frac{\partial u(t,x)}{\partial t}=\frac{1}{2}\Delta u(t,x)+u(t,x)\frac{\partial^{d+1}W}{\partial t\partial x_1\ldots\partial x_d}(t,x), \\
		u(0,x)=f(x), \\
	\end{array}
	\right.
\end{equation}
where   $B_t^x:=x+B_t$  denotes   the Brownian motion starting at $x$,
and  $\delta (\cdot)$
denotes the Dirac delta function.

%

The above  work has been extended to   more general Gaussian noises and  is used to study the asymptotics of the solution
(e.g. \cite{bc,ch17,ch16, chsx,Hu19,Hu,Hu1}  and the references therein).  There are also extension of this work to fractional Laplacian (generator of the stable processes \cite{chss}).

However,  in  applications or in the  theory of partial differential equations, the equation is in a more general form, namely,
  the Laplacian in \eqref{sec1-eq5} is usually replaced by a (time dependent) elliptic operator  and  the Euclidean space $\RR^d$ is replaced by a general (bounded or unbounded)  domain.

To  describe the objective of this work,
let $D$ be a bounded domain of the $d$-dimensional Euclidean space
$\RR^d$ with  Lipschitz
 continuous boundary  $\partial D$.
Denote  the second order partial differential operator $L_t$:
\begin{equation}\label{e.1.defL}  L_t:=\frac{1}{2}\sum^d_{i,j=1}(\sigma\sigma')_{ij}(t,\cdot)\frac{\partial^2}{\partial{x_i}\partial{x_j}}+\sum^d_{i=1}b_i(t,\cdot)\frac{\partial}{\partial{x_i}},
\end{equation}
where  $\sigma: \mR_+\times D\rightarrow \mR^d\otimes\mR^d$, $b : \mR_+\times D\rightarrow\mR^d$ are measurable functions
satisfying some conditions which will be specified later  and  $\sigma'$ denotes  the transpose of $\sigma$.

We shall obtain the Feynman-Kac formula for
the following stochastic partial differential equation with fractional noise
\begin{equation}\label{sec1-eq4}
	\left\{
	\begin{array}{lll}
		\frac{\partial u(t,x)}{\partial t}=L_tu(t,x)+u(t,x)\dot{W}(t,x),~~~(t,x)\in [0, \infty) \times D,\\
		u(0,x)=f(x), ~~~x\in D,\\
		u(t,x)=g(t,x),~~~(t,x)\in  [0, \infty) \times\partial D,
	\end{array}
	\right.
\end{equation}
   where we assume  $g(0,x)=f(x)$ for $x\in \partial D$, and where $W(t,x)$ is a fractional Gaussian noise
with Hurst parameters $H_0$ in time and $(H_1,\ldots,H_d )$ in space, respectively,  namely,
denoting   $
\phi_H(x)  =H(2H-1)   |x |^{2H-2}    \,,\  \forall \  H\in (0, 1/2)\cup (1/2, 1)\,, \ x\in \RR$.
\begin{equation}
	\left\{
	\begin{array}{lll}
	 \EE \left[\dot{W}(t,x)\right]=0\,;\\
	  \EE \left[\dot{W}(t,x)  \dot{W}(s,y)\right] =
	  \phi_{H_0}(t-s) \prod_{i=1}^d  \phi_{H_i}(x_i,y_i)  \,,
	\end{array}
	\right.
\end{equation}
$x, y\in \mR^d$ and where the product
$u(t,x)\dot{W}(t,x)$ in  \eqref{sec1-eq4} can be both in the pathwise (Stratonovich) sense or  in the   sense of Wick product  (e.g. \cite{HY}.
In other word, we use Skorohod integral in the associated mild form).

Since the coefficients in  \eqref{sec1-eq4} are time dependent,  even the classical Feynman-Kac formula
for the solution is a little more complicated. First, we need to
appropriately associate a stochastic differential
equation (SDE)  corresponding  to  $L_t$ for the purpose of Feynman-Kac formula. To this end,  we extend the
definition domain of the  coefficients $\sigma(t,x)$ and $b(t,x)$ from $t\in \RR_+$ to $t\in \RR$  by
\[
\sigma(-t,x)=\sigma(t,x)\,,\quad b(-t,x)=b(t,x)\,, \quad \forall \ t\ge 0\,.
\]
We also extend definition domain of the  coefficients $\sigma(t,x)$ and $b(t,x)$ from $x\in D$ to $x\in
\RR^d$ in a  way which will be defined in the next section.
With these extension  we consider for any  fixed  $t>0$
  the following SDE:
\begin{equation}\label{sec1-eq2}
	dX_s^{t,x}=\sigma(t-s,X_s^{t,x})dB_s+b(t-s,X_s^{t,x})ds,\quad s\ge 0\,, ~~~X_0^{t,x}=x\in D\,,
\end{equation}
where $B_\cdot $ is a $d$-dimensional Brownian motion independent of $W$.

Let $\tau^{t,x}_D:=\inf\{s>0:X_s^{t,x}\in\partial D\}$  be the first exit time of process $X^{t,x}_\cdot$ from the domain $D$. In the following,   for simplicity of notations, we shall sometimes drop  the superscripts $t,x$ in  $\tau^{t,x}_D$ when there is no confusion,  namely, we sometimes write  $\tau_D=\tau_D^{t,x}$. To obtain heuristically our Feynman-Kac formula candidate let us recall now a
known result. When the noise $\dot{W}(t,x)$ in \eqref{sec1-eq4} is
a usual  continuous function $c(t,x)$, namely, for
the following second order  partial differential equation (PDE)
\begin{equation}\label{sec1-eq1}
	\left\{
	\begin{array}{lll}
		\frac{\partial u(t,x)}{\partial t}=L_tu(t,x)+ u(t,x) c(t,x),~~(t,x)\in[0,T]\times D,\\
		u(0,x)=f(x), ~~~x\in D,\\
		u(t,x)=g(t,x),~~~(t,x)\in[0,T]\times\partial D,
	\end{array}
	\right.
\end{equation}
where  $c(t,x)$ is a nice (e.g. continuous, bounded) deterministic function,   then
by \cite[p.133, Theorem 2.3]{Freidlin}, when $\sup_{x\in D}\mE\left[ \tau^{t,x}_D\right]<\infty$, 
we have the well-known Feynman-Kac formula to the Cauchy-Dirichlet problem (\ref{sec1-eq1}),
\begin{eqnarray}
	u(t,x)& =&\mE\Bigg[ f(X^{t,x}_t)\mI_{t\le  \tau_D}\exp\Big\{\int_0^tc(t-s,X_s^{t,x})ds\Big\} \nonumber\\
	&&\qquad+g(\tau_D,X^{t,x}_{\tau_D})\mI_{t>\tau_D}
	\exp\Big\{\int_0^{\tau_D}c(t-s,X_s^{t,x})ds\Big\}\Bigg]  \nonumber\\
	& =&\mE\Bigg[ f(X^{t,x}_t)\mI_{t\le \tau_D}\exp\Big\{\int_0^tc( s,X_{t-s}^{t,x})ds\Big\} \nonumber\\
	&&\qquad+g(\tau_D,X^{t,x}_{\tau_D})\mI_{t>\tau_D}\exp\Big\{\int_ {t-\tau_D}
	^t c( s,X_{t-s}^{t,x})ds\Big\}\Bigg] \,,
	\label{sec1-eq3}
\end{eqnarray}
where $\mI_A$ denotes the indicate function of $A$.
Motivated by
\eqref{sec1-feynman},  \eqref{sec1-eq3} and
when the product $u\dot W$ is pathwise one
(Stratonovich integral),  we can write heuristically  a candidate for  the Feynman-Kac
representation  for the solution
to \eqref{sec1-eq4} as
\begin{equation}
	\begin{cases}
		u(t,x)     = \mE^B\Big(f(X^{t,x}_t)\mI_{t\le \tau_D}\exp\Big\{A(t, t ,x) \Big\}\\
		\qquad\qquad \qquad\qquad +g(\tau_D,X^{t,x}_{\tau_D})\mI_{t>\tau_D}\exp\Big\{A(\tau_D, t ,x)\Big\}\Big)\,;
		\\ \\
		A(r, t ,x)  =\int_{t-r}^t   W(X_{t-s}^{t,x} ,ds) =\int_{t-r}^t \int_{\mR^d}\delta (X_{t-s}^{t,x}-y)W(dy,ds) \,,
	\end{cases}
	\label{sec1-eq04}
\end{equation}
where  
the process $X_\cdot ^{t,x}$ is the solution of equation (\ref{sec1-eq2});
and $\EE^B$ denotes the expectation with respect to the Brownian
motion $B$, namely with respect to the process $X^{t,x} $.

When  the product $u\dot W$ is the Wick product
(Skorohod integral)  the  candidate for  the Feynman-Kac
representation  for the solution
to \eqref{sec1-eq4} can be written as  		\begin{equation}
	\begin{cases}
		u(t,x)     = \mE^B\Big(f(X^{t,x}_t)\mI_{t\le \tau_D}\exp\Big\{A(t, t ,x) \Big\}\\
		\qquad\qquad \qquad\qquad +g(\tau_D,X^{t,x}_{\tau_D})\mI_{t>\tau_D}\exp\Big\{A(\tau_D, t ,x)\Big\}\Big)\,;
		\\ \\
		A(r, t ,x)  =\int_{t-r}^t   W(X_{t-s}^{t,x} ,ds)\\
		\qquad\qquad \qquad\qquad -\frac12 \int_{t-r}^t\int_{t-r}^t
		|s_2-s_1|^{2H_0-2}   \phi_H( X_{t- s_2}^{t,x}-X_{t- s_1}^{t,x}) ds_1ds_2\\
			\qquad\qquad =\int_{t-r}^t \int_{\mR^d}\delta (X_{t-s}^{t,x}-y)W(dy,ds)\\
		\qquad\qquad \qquad\qquad -\frac12 \int_0^r \int_0^r
		|s_2-s_1|^{2H_0-2}   \phi_H( X_{  s_2}^{t,x}-X_{  s_1}^{t,x}) ds_1ds_2\,.
	\end{cases}
	\label{sec1-eq04a}
\end{equation}
The term $\frac12 \int_0^r \int_0^r
		|s_2-s_1|^{2H_0-2}   \phi_H( X_{  s_2}^{t,x}-X_{  s_1}^{t,x}) ds_1ds_2$ appeared in 	\eqref{sec1-eq04a}
		is because of the use of the Wick product
		in \eqref{sec1-eq4}.

%

As in the literature, to justify \eqref{sec1-eq04}
and \eqref{sec1-eq04a}, the following three tasks must be completed.
\begin{enumerate}
	\item[(i)] Show that the stochastic integral $\int_{t-r}^t \int_{\mR^d}\delta (X_{t-s}^{t,x}-y)W(dy,ds)$
	is well-defined  which requires some work    since $\delta$ is a generalized function. But as usual this problem is not much difficult.
	\item[(ii)]  Show the above integral is exponentially integrable, which is   a difficult problem.
	\item[(iii)]  Show that   the well-defined $u(t, x)$ by \eqref{sec1-eq04} is a mild  solution to
\eqref{sec1-eq4}, which is also complicated.
\end{enumerate}

A main objective of this work is to complete the above three  tasks under some broad assumptions on the
domain $D$, coefficients $\sigma$ and $b$, and the Hurst parameter range.
In the precedent cases of  stochastic heat equation, the associated process $X_\cdot^{t,x}=B_t+x$
is the Brownian motion, which has the independent increment property. This property
was   ingeniously exploited in the previous work to complete the above
three tasks. However, in our new  situation, the
underline process $X_\cdot ^{t,x}$ does not have this  independent increment property so that we cannot apply the previous work directly.  We shall use mainly  the following
two properties of the process $X_\cdot ^{t,x}$: its Markov property
and the Aronson type of estimates of the associated transition probability density.

As an  important application  of our Feynman-Kac formula
we shall prove  the sharp lower bounds for the  moment of the solution
to \eqref{sec1-eq4},    in terms of the exact asymptotics as $t$ goes to infinity or as the order of the
moment goes to infinity.   This sharp lower bound seems hard to obtain by using other method.
To obtain this sharp lower moment bound  we need further to establish  a small ball like estimate for our stochastic differential equation \eqref{sec1-eq2}.  This kind of small ball result has not been investigated before and has  its own interest.

After the completion of this work we came across an interesting paper \cite{SWY}  which also dealt with Feynman-Kac formula
with general elliptic operators in whole space $\RR^d$. They assume that the transition probability density satisfies the
heat kernel estimate.

The paper is organized as follows. Section 2 contains some preliminaries and notations on the fractional noise $W$. We also introduce some notations that we are going to use
in the remaining  parts  of the paper. In Section 3, we state some results and properties on the  fundamental solution of partial differential equation and obtain the Aronson   type  estimates
for  solution to $d$-dimensional SDE  in bounded domain $D$. The results are not
easily found in standard literature.  We tailor  them in the form that we are going to use
in this paper.  In Section 4, we give a definition on   the nonlinear stochastic integral appeared in equation (\ref{sec1-eq04}) by using smooth approximation and provide the existence and   exponential integrability of the stochastic Feynman-Kac  additive functional, a critical  random variable $V_{t, \tau_D,x}$ defined in (\ref{sec4-eq4}).  The first two mentioned tasks will be completed in this section.
  We prove   that the Feynman-Kac
representation (\ref{sec1-eq04}) is a mild solution of (\ref{sec1-eq4}) in  Section 5, completing the third task. This  paper includes two  applications of this newly established stochastic Feynman-Kac formula.  One is about the
H\"older continuity of solution  of   (\ref{sec1-eq4}).
In the previous case, the independent increment and the scaling properties of the Brownian motion
is cleverly used. But the  general diffusion processes
do  not possess these two properties. This poses a very serious challenging  since we need to compare the solution of different stochastic differential equations and we need to handle the singularity.  We particularly introduced a new technique to handle the singularity of integrals.
Due to the limitation of our technique, we obtain a H\"older continuity weaker than in the case of
stochastic heat equation (Laplacian/Brownian motion  case).

Another application  is about matching lower and upper  moment bounds
of the moments of the solution, solving the issue of intermittency. This is presented
in Section 6.  It seems  that the Feynman-Kac representation
is  the only effective way to obtain the matching lower moment bounds
  and  the classical
chaos expansion method can  give the   lower moment bound  only  for second moment (see however, \cite{hw}  for some other  cases).
For example, near the completion of this work, we encounter a work
\cite{ccl},  where the upper moments are obtained for all the
moments of the solution.  However,  the lower counterpart is only known for the second moment.
By using the Feynman-Kac formula we obtained, we can provide
the matching lower moments for all moments.

Throughout this paper, we use the following convention: $C$ with or without subscripts will denote a positive constant, whose value may vary  in different places, and whose dependence on the parameters can be traced from the calculations.  $\sigma'$ denotes the transport of matrix $\sigma$ and
 the matrix norm $\|\sigma\|^2:=\tr(\sigma\sigma')=\sum_{i,j}\sigma^2_{ij}$
 is the Hilbert-Schmidt norm. $|\cdot|$ denotes the Euclidean vector norm.   $\Gamma(\cdot)$  denotes the  Gamma function and ${\bf B}(\cdot,\cdot)$ denotes Beta function.
 Without confusion, $\delta$ is used to denote a  positive constant,  the Dirac delta function,
 and  skorokhod integral, depending on their appearances. 
We define $C^{1,1}$ to be the set of Lipschitz functions that have Lipschitz first partial derivatives. We say that $D$ is a $C^{1,1}$ domain if locally the boundary of $D$ is the graph of a $C^{1,1}$ function.


\section{\bf Aronson  type estimates of the solution to a $d$-dimensional SDE    in  bounded domain}
\label{s.2}

\subsection{Lipschitz domain $D$}
First we  recall  the definition of   Lipschitz domain $D\subseteq \RR^d$.
A   domain $D$ in $\mR^d, d\geq 2$, is called a Lipschitz domain if there exist positive constants $k$ and $r_0$ such that for every point $q\in \partial D$, there exist a function
$f_q:\mR^{d-1}\rightarrow \mR$  
such that
\begin{enumerate}
	\item[(i)]   $|f_q(x')-f_q(y')|\leq k|x'-y'|$ for all $x', y'\in \mR^{d-1}$;
\item[(ii)]   
Let $B(q,r_0)$ be the ball of $\RR^d$ with center $q\in \RR^d$ and radius $r_0$. Then
\ce
D\cap B(q,r_0)=B(q,r_0)\cap \{y=(y',y_n), y_n>f_q(y')\}
\de
and
\ce
\partial D\cap B(q,r_0)=B(q,r_0)\cap \{y=(y',y_n), y_n=f_q(y')\}.
\de
\end{enumerate}
  $k$ is called the Lipschitz constant of $D$ and $r_0$ is called the localization radius of $D$.

\subsection{The extensions of coefficients $b$ and $\sigma$}\label{s.2.2}
  We assume that $\sigma$ and $b$ are defined
on a bounded domain $D$ and we need to extend them to $\RR^d$, keeping some desired properties.

By a modulus of continuity we mean a function $\omega: [0,+\infty)\rightarrow \mR$ such that $\omega(0)=0$
and that  $\omega$ is continuous, nondecreasing, and semiadditive on $[0,+\infty)$, i.e.,
\ce
\omega(z_1+z_2)\leq \omega(z_1)+\omega(z_2), \quad 0\leq z_1, z_2<+\infty.
\de
We denote by  $C^{0,\omega}([0,T]\times D)$ the space  of bounded continuous   functions $f:[0,T]\times D\rightarrow \mR$ with the  norm
\ce
\|f\|_{0,\omega}=\sup_{[0,T]\times D}|f(t,x)|+
\sup_{(t,x), (t+\De t, x+\De x)\in [0,T]\times D}\frac{\Delta_{t,x}f(t,x)}{\omega(|\Delta x|+|\Delta t|^{1/2})},
\de
where we use the standard notation: $\Delta x=(\Delta x_1, \Delta x_2, \ldots, \Delta x_d)$ , $\Delta_{x}f(t,x)=f(t, x+\Delta x)-f(t,x)$, $\Delta_{t}f(t,x)=f(t+\Delta t, x)-f(t,x)$, $\Delta_{t,x}f(t,x)=f(t+\Delta t, x+\Delta x)-f(t,x)$.
Similarly, we can also introduce the notation  $C^{0,\omega}([0,T]\times \RR^d)$.


With the above notation, we assume that the coefficients $b$ and $\si$  are real-valued functions in $D$
satisfying
following conditions.

\medskip
\noindent Hypothesis (({\bf H2.1})):
\begin{enumerate}
	\item[(i)] There is $\delta>0$ such that
	$\xi^{'} a(t,x)\xi\geq \delta|\xi|^2  $ for   all $(t,x)\in[0,T]\times D$, $\xi\in \mR^d$.
	\item[(ii)]   $a_{ij}$ is Dini continuous, namely,
	$a_{ij}\in C^{0,\omega_0}([0,T]\times D)$, $i,j=1,2,\ldots,d$, where $\omega_0$ is a modulus of continuity satisfying the double Dini condition
	\ce
	\int_0^z\frac{1}{y}\int_0^y\frac{\omega_0(\xi)}{\xi}d\xi<\infty, \quad \forall \ z>0
	\de
	and for some $C>0$ and $\delta\in(0,1)$,
	\ce
	\frac{\omega_0(z_1)}{z^{\delta}_1}\leq C\frac{\omega_0(z_2)}{z^{\delta}_2},\quad  \forall z_1\geq z_2>0.
	\de
	\item[(iii)]
	$b_i\in C^{0,\omega_1}([0,T]\times D)$, $i=1,2,\ldots,d$, where $\omega_1$ is a modulus of continuity satisfying the Dini condition
	\ce
	\int_0^z\frac{\omega_1(\xi)}{\xi}d\xi<\infty, \quad  \forall z>0.
	\de
\end{enumerate}


\begin{lemma}\label{l.2.1}
	Let $D$ be a bounded Lipschitz domain in $\mR^d$  and  $f(t,x)\in C^{0,\omega}([0,T]\times D)$.  Then there is a real valued function $\widetilde {f}(t,x)\in  C^{0,\omega}([0,T]\times \mR^d)$
	such that its restriction     on $[0,T]\times D$
	is $f(t,x)$.
\end{lemma}
\begin{proof}
	We shall show that the function
	\ce
	\widetilde{f}(t,x):=\inf_{y\in D}\{f(t,y) +\omega(|x-y|)\}
	\de
	is the desired extension of $f(t,x)$.
	
		\noindent Step 1:  If $x,y\in D$, then
	\ce
	f(t,x)+\omega(|x-x|)=f(t,x)
	\de
	and
	\ce
	f(t,y)+\omega(|x-y|)\geq f(t,x).
	\de
	Consequently,
	\ce
	\widetilde{f}(t,x)=\inf_{y\in D}\{f(t,y) +\omega(|x-y|)\}=f(t,x) ~~for ~~~x\in D
	\de
	and so $\widetilde{f}(t,x)$ is  well-defined for $x\in D$
	and is   equal to  $f(t,x)$ on $D$.
	
	\smallskip
	\noindent Step2:
Let $x$ be  any point outside $D$
and let $y, z$ be any two points of $D$.   Then
		\ce
		f(t,y)-f(t,z)\leq \omega(|y-z|)\leq \omega(|y-x|)+\omega(|x-z|)\,.
		\de
This shows that $f(t,z)+\omega(|x-z|)$ is bounded
below in $D$.  Thus
	\ce
\tilde f(t,x):=	  \inf_{z\in D}\{f(t,z)+\omega(|x-z|)\}
\de
is well-defined for any $x\in \RR^d$ and
		\ce
		\sup_{y\in D}\{f(t,y)-\omega(|y-x|)\}\leq \tilde f(t,x)=\inf_{z\in D}\{f(t,z)+\omega(|x-z|)\}\,.
		\de
%
%

	\noindent Step 3:  We shall now prove that $\widetilde{f}(t,x)\in C^{0,\omega}([0,T]\times\mR^d)$. Let $x_1$ and $x_2$ be two
	points of $\mR^d$, $t_1,t_2$ be two points in $[0,T]$. For a given $\varepsilon>0$ we may choose a $y\in D$ so that
	\ce
	\widetilde{f}(t_1,x_1)\geq f(t_1,y)+\omega(|y-x_1|)-\varepsilon
	\de
	hold. On the other hand,
	\[
	\widetilde{f}(t_2,x_2)\leq f(t_2,y)+\omega(|y-x_2|) \,.
	\]
	Then we have
	\ce
	\begin{split}
 \widetilde{f}(t_2,x_2)-&\widetilde{f}(t_1,x_1) \\
	&\leq  f(t_2,y)+\omega(|y-x_2|)-(f(t_1,y)+\omega(|y-x_1|)-\varepsilon)\\
	&= f(t_2,y)-f(t_1,y)+\omega(|y-x_2|)-\omega(|y-x_1|)+\varepsilon \\
	&\leq  |f(t_2,y)-f(t_1,y)|+ \omega(|x_2-x_1|)+\varepsilon\\
	&\le \omega_0(|t_2-t_1|)+\omega(|x_2-x_1|)+\varepsilon  \,,
	\end{split}
	\de
	where the last inequality follows from the sub-additivity of
	$\omega$.
	Similarly, we have
	\ce
	\begin{split}
	\widetilde{f}(t_1,x_1)- \widetilde{f}(t_2,x_2)
	&\le \omega_0(|t_2-t_1|)+\omega(|x_2-x_1|)+\varepsilon.
	\end{split}
	\de
	Since $\varepsilon>0$ is arbitrary, we see   indeed that $\widetilde{f}(t,x)\in C^{0,\omega}([0,T]\times\mR^d)$.
	Thus the proof of this lemma is complete.
\end{proof}

We recall a definition of regular boundary point of $D$ from  probabilistic  analysis.
	\begin{definition}\label{def3-1} A point $z$ is called a regular boundary point of $D$ (associated with \eqref{sec1-eq2})
		if   for any $z\in \partial D$   we have  $\mP^z(\tau_D=0)=1$.  The set of regular boundary points of $D$ is denoted by $(\partial D)_r$ and $D$ is said to be regular if and only if $\partial D=(\partial D)_r$.
	\end{definition}
	
	A sufficient condition for regularity, known as the cone condition, is given in the
	following proposition in \cite[Proposition 1.22]{Chung}.
	\begin{proposition}\label{prop3-01}
		Let $z\in \partial D$ be a boundary point of $D$. If there exists a cone $A$ with vertex $z$ such that $A\cap B(z,r)\subseteq D$ for some $r>0$, then $z$ is regular.
	\end{proposition}

\subsection{Aronson estimate}
Let $(\Omega,\cF,\mP,\{\cF_s\}_{s\geq0})$ be a complete filtered probability space
with a filtration satisfying the usual condition and let $B=\{B^i\}^d_{i=1}$ be a $d$-dimensional Brownian motion.
Assume that  $b:[0, T]\times D \to \RR^d$, $\sigma:[0, T]\times D\to \RR^d\times \RR^d $ 
satisfy    ({\bf H2.1})  and they are extended to the whole space $\RR^d$ via Lemma \ref{l.2.1}.
To obtain Feynman-Kac representation \eqref{sec1-eq04} for the solution of
equation \eqref{sec1-eq4}
  we  consider the     stochastic differential equation \eqref{sec1-eq2} which we recall here:
\begin{equation}\label{sec2-eq01}
dX_s^{t,x}=b(t-s,X_s^{t,x})ds+\sigma(t-s,X_s^{t,x})dB_s,\quad 0\le s <\infty\,, \quad X_0^{t,x}=x\,,
\end{equation}
where    $t\geq0$ and $x\in D$  are given,   $b=(b_1,\ldots,b_d):\mR_+\times\mR^d\rightarrow \mR^d$ and $\sigma=(\sigma_{ij})^d_{i,j=1}:\mR_+\times\mR^d\rightarrow \mR^d\times\mR^d$. The equation (\ref{sec2-eq01})
has a unique solution under some regular conditions on $b, \sigma$ for all $s\ge 0$
if we extend $b$ and $\si$  from $t\in \RR_+$  to  all $t\in \RR$  by defining $b(-s, x)=b(s, x)$ and $\si(-s, x)=\si(s, x)$
for all $s\in \mR_+$.  The solution $(X_s^{t,x}, s\ge 0)$ is a Markov process
but it is not homogeneous in time. 

Some readers may be annoyed  by the appearance of $t$ in \eqref{sec2-eq01}.
To get rid of the $t$ there we can   augment the dimension by considering the equivalent SDE
\begin{equation}\label{sec2-eq2}
\left\{
                  \begin{array}{lll}
                   d\xi(s)=-ds, \\
                   dX^{t,x}_s=b(\xi(s),X^{t,x}_s)ds+\sigma(\xi(s),X^{t,x}_s)dB_s,
                   \end{array}
                    \right.
\end{equation}
with $(\xi(0),X^{t,x}_0)=(t,x)$.  
The process $(\xi(s),X^{t,x}_s)_{s\geq 0}$ is a strong homogeneous Markov process
 (e.g.
\cite[Theorem 4.6, Chapter 5]{Durrett}).
%

Denote $a(t,x)=(a_{ij}(t,x))_{1\le i,j\le d}:=\sigma(t,x)\sigma'(t,x)$ for $(t,x)\in[0, T]\times \mR^d$.
The parabolic partial differential equation associated to (\ref{sec2-eq01}) is
\begin{equation}
\partial_tu=L_tu\label{e.sec2-3}
\end{equation}
where
\ce
L_tu:=L_{t,x}u=\sum^d_{i=1}b_i(t,x)\partial^iu+\frac{1}{2}\sum^d_{i,j=1}a_{ij}(t,x)\partial^{ij}u\,,
\de
where $\partial ^ i=\partial ^ i_x=\frac{\partial}{\partial x_i}$ and $\partial^{ij}=
\partial^{ij}_x=\frac{\partial^2}{\partial x_i\partial x_j}$.

From   \cite[Theorem 1.1]{Zhenyakova}, it follows   that  the fundamental solution $p(\tau,\xi;s,y)$ to the equation $\partial_tu=L_tu$ on $[0,T]\times \mR^d$ exists uniquely in the classical sense. It is well-known that the fundamental solution to
\eqref{e.sec2-3} is the transition probability density function of $X_s^{t,x}$.
More specifically,
we have the following estimate about the fundamental solution.
\begin{proposition}\label{prop3-1}
Assume that the conditions (i)-(iii) in (({\bf H2.1}))  for the coefficients $a,b$ are satisfied with $D$ replaced by $\mR^d$.  Then the fundamental solution $p( s,y; t,x) $
 to the equation $\partial_tu=L_{t,x} u$ on $[0,T]\times \mR^d$ exists uniquely and
for any nonnegative integer  $\ell$ and nonnegative multi-index $k=(k_1, \cdots, k_d)$, it satisfies the estimates
\begin{equation}\label{sec3-eq3a}
|\partial^\ell _t \partial^k_x  p( s,y;t,x)|\leq C(t-s)^{-\frac{(2\ell +|k|+d)}{2}}\exp\Big\{-\frac{c|x-y|^2}{2(t-s)}\Big\},~~~2\ell +|k|\leq 2\,
\end{equation}
for $(s, y), ( t, x)\in[0, T ]\times \mR^d$ such that $s< t$,  where  $|k|=k_1+\cdots+k_d$ and $C, c$ denote positive constants depending on $d,T$ and the coefficients of the operator of $L_{t,x} $.
\end{proposition}

With these bounds we proceed to show that the transition semigroup $\{P_{\tau}\}$ defined by
\ce
P_{\tau}f(\xi):=\int_{\mR^d}f(y)p(\tau,\xi;s,y)dy
\de
has the strong Feller property.
\begin{proposition}\label{prop3-2}
For each $\tau>0$, $f\in L^{\infty}(\mR^d)$, we have $P_{\tau}f(\xi)\in C(\mR^d)$.
\end{proposition}
\begin{proof}
Let $f\in L^{\infty}(\mR^d)$  and $r>0$. If $\xi_n\rightarrow \xi$, then
\ce
 |P_{\tau}f(\xi_n)-P_{\tau}f(\xi)|&\leq& \lim_{n\rightarrow 0}\|f\|_{L^{\infty}(\mR^d)}\int_{\mR^d}|p(\tau,\xi_n;s,y)-p(\tau,\xi;s,y)|dy\\
&=& \|f\|_{L^{\infty}(\mR^d)}\int_{B(0,r)^c}|p(\tau,\xi_n;s,y)-p(\tau,\xi;s,y)|dy\\
&&\qquad+ \|f\|_{L^{\infty}(\mR^d)}\int_{B(0,r)}|p(\tau,\xi_n;s,y)-p(\tau,\xi;s,y)|dy\\
&\leq& \|f\|_{L^{\infty}(\mR^d)}\int_{B(0,r)^c}|p(\tau,\xi_n;s,y)|+|p(\tau,\xi;s,y)|dy\\
& & \qquad+ \|f\|_{L^{\infty}(\mR^d)}|\xi_n-\xi|\int_{B(0,r)}|\partial_{\xi}p(\tau,\xi';s,y)|dy\\
&=&I_{1,n,r}+
I_{2,n,r}
 \,.
\de
For small $\varepsilon>0$ one can choose $r_\varepsilon>0$ such that
$I_{1,n,r_\varepsilon}<\varepsilon$ by  the estimate   (\ref{sec3-eq3a}). For any $r_\varepsilon>0$ we also see that
$\lim_{n\to \infty} I_{2,n,r_\varepsilon}=0$.  This proves that   $\lim_{n\to \infty} |P_{\tau}f(\xi_n)-P_{\tau}f(\xi)|=0$,
completing the proof of the proposition.
\end{proof}
With  Proposition \ref {prop3-1}, we shall prove  the existence and uniqueness of the solution to equation \eqref{sec1-eq2}. Noticing that in Equation \eqref{sec1-eq2},  $t$ is a given fixed time parameter and the solution, if exists, would be a time inhomogeneous (time dependent) Markov process. We are interested in its transition probability $P(X^{t,x}_r\in A|X^{t,x}_s=y)$.  Since in this Markov process $t$ is fixed, to simplify notation we shall use another notation $\tilde t$ to replace and we shall also omit the dependence of the Markov process on $\tilde t$. Thus,   we shall consider the  following stochastic differential  equation
\begin{equation}\label{sec3-eq4}
dX^{ t,x}_r=\tilde \sigma(r,X^{t,x}_r)dB_r+\tilde b(r,X^{t ,x}_r)dr,~~ X^{t,x}_t=x,
\quad 0\le t\le r<\infty\,,
\end{equation}
for all $x\in\mR^d$,     where $\tilde b(r,x)=
b(\tilde t-r, x)$ and $\tilde \si(r,x)=
\si(\tilde t-r, x)$, $\tilde t$ is fixed. With these notations,
$P(X^{\tilde t,\tilde x}_r\in A|X^{\tilde t,\tilde x}_t=x)=P(X^{  t, x}_r\in A )
$. It is clear due to Markov property, this probability is independent of $\tilde x$.  The corresponding    generator
associated with \eqref{sec3-eq4}
is
 \ce
\tilde L_tu:=\tilde L_{t,x}u=\sum^d_{i=1}\tilde b_i(t,x)\partial^iu+\frac{1}{2}\sum^d_{i,j=1}\tilde a_{ij}(t,x)\partial^{ij}u\,.
\de
By Proposition \ref{prop3-1}, we know that  $\partial_tu=L_tu$  has a unique $p(  t,x; s,y )$, $0\le t\le s$, $x,y\in \mR^d$. It is easy to verify that
$\{ p(  t,x; s,y ),0\le t\le s, x,y\in \mR^d\}$ constructs a Markov transition probability density, and it gives a Markov
process $X_{\cdot}^{ t,x}$.  It is well-known that this Markov process is a solution to the martingale problem.
This means that for any $f\in C_b^{1,2}(\RR_+\times \RR^d, \RR)$,
\[
M_s^f:=f(s, X^{t,x}(s) )-\int_t^s \left[ \frac{\partial f}{\partial
r} -L_{r,x} \right]f (r,
X^{t,x}(r)) dr\,, \quad s\ge t\
\]
is a martingale (as a process of $s$, e.g. \cite{kurtz}).
Thus,   $X_{\cdot}^{ t,x}$ is the   unique
weak solution to \eqref{sec3-eq4}
(e.g. \cite{Stroock}).

We emphasize that the first two parameters in $p(t , x; s,y)$  are starting ones and the last two parameters are the arrival ones. Namely,
the transition probability density function of
$X_{\cdot}^{ t,x}$  exists, and it coincides with $p$:
 \begin{equation}
p(t , x; s,y) =
\PP \left( X_s^{ t,x}=y|X_t^{ t, x}=x\right)=\PP \left( X_s^{ t,x}=y \right)\,.
\end{equation}

Next,  we need to study  the  fundamental solution $p_D(\tau,\xi;s,y)$ of  parabolic partial differential equations  on bounded domain $D$ and
some associated heat kernel estimates when   the coefficients satisfy  the Dini conditions on bounded Lipschitz domain $D$.

For any given $T>0$,  by  \cite[Theorem  0.25]{Varopoulos} we have  on $(0, T]\times D\times D$
\begin{equation}\label{sec3-eq05}
p_D(\tau,\xi;s,y)\asymp p(\tau,\xi;s,y)\mP^{\xi}(\tau_D>s)\mP^y(\tau_D>s) \,,
\end{equation}
where $\tau_D$ is the first exit time from $D$ of the process $X_{\cdot}^{t,x}$, i.e., $\tau_D:=\inf\{s>0, X_s^{t,x}\notin D\}$ where $X_s^{t,x}$ is the solution of (\ref{sec1-eq4}) and $\mP^{\xi}$ is a probability from original state $\xi$ on $(\Omega, \cF, \mP^{\xi})$ .

%

Given domain  $D \subset \mR^d$, let $X^{t,x}_s$ be the solution of (\ref{sec3-eq4}), we define killed process
\ce
X^{t,x}_{s,D}:=
\left\{
        \begin{array}{lll}
        X^{t,x}_s, &\qquad if \qquad s<\tau_D,\\
        \partial,  &\qquad if \qquad s\geq \tau_D,
        \end{array}
        \right.
\de
where $\partial$ is a cemetery state.
For any $B\in\mathfrak{B}(\mR^d)$,   we have
\ce
\int_Bp(t,x;s,y)dy&=&\mP^x(X^{t,x}_s\in B)\\
&=&\mP^x(X^{t,x}_s\in B, s<\tau_D)+\mP^x(X^{t,x}_s\in B, s\geq \tau_D)\\
&=&\mP^x(X^{t,x}_s\in B, s< \tau_D)+\mE^x\left(\mP^{X^{t,x}_{\tau_D}}(X^{t,x}_{s-\tau_D}\in B), s\geq \tau_D\right)\\
&=&\mP^x(X^{t,x}_s\in B, s< \tau_D)+\mE^x\left(\int_Bp(t,X^{t,x}_{\tau_D};s-\tau_D,y)dy,s\geq\tau_D\right).
\de
%

For $s>0$, $x,y\in\mR^d$, set
\ce
K_D(t,x;s,y):=\mE^x[   p(t, X^{t,x}_{\tau_D}; s-\tau_D, y), \tau_D\leq s]
\de
and
\begin{equation}\label{sec2.def.8}
p_D(t,x;s,y):=p(t,x;s,y)-K_D(t,x;s,y).
\end{equation}

\begin{theorem}\label{th3-1} Let $D$ be a bounded Lipschitz domain in $\mR^d$.  Then for any $s>0$,
\begin{equation}\label{sec3-eq9}
\mP^x(s< \tau_D, X^{t,x}_s\in A)=\int_Ap_D(t,x;s,y)dy,~~x\in\mR^d,~~~A\in \cB(\mR^d)\,,
\end{equation}
i.e. $p_D(t,x;s,y)$ is the transition density.

The function $p_D(t,\cdot;s,\cdot)$ is jointly continuous and strictly positive on $D\times D$.
For any $t> s>r>0$,\ $x, y\in D$, we have
\begin{equation}\label{sec3-eq10}
p_D(t,x;s,y)=\int_Dp_D(t,x;r,z)p_D(r,z;s,y)dz.
\end{equation}
\end{theorem}
\begin{proof}
For any $s>0$,   $x\in \mR^d$,  and $A\in \cB(\mR^d)$
\ce
\mP^x(s<\tau_D, X^{t,x}_s\in A)&=&\mP^x(X^{t,x}_s\in A)-\mP^x(\tau_D\leq s, X^{t,x}_s\in A)\\
&=&\int_Ap(t,x;s,y)dy-\mP^x(\tau_D\leq s, X^{t,x}_s\in A)\,.
\\
&=& \int_Ap(t,x;s,y)dy-\mP^x(\tau_D< s, X^{t,x}_s\in A)
\de
because $\mP^x(\tau_D=s)=0$   for any $s> 0$.
For any $0<u<s$, $n\geq 1$, and $1\leq k\leq 2^n$, set
\ce
T_n=
\left\{
        \begin{array}{lll}
        \frac{ku}{2^n}, \qquad if ~~~\frac{(k-1)u}{2^n}\leq \tau_D<\frac{ku}{2^n},\\
        \infty,       \qquad  if ~~\tau_D\geq u.
        \end{array}
        \right.
\de
Then by the Markov property of the solution $X^{t,x}$ to (\ref{sec3-eq4}), we have
\begin{eqnarray}\label{sec3-eq11}
\mP^x(\tau_D< u, X^{t,x}_s\in A)&=&\sum^{2^n}_{k=1}\mP^x\left(\frac{(k-1)u}{2^n}\leq \tau_D<\frac{ku}{2^n}, X^{t,x}_s\in A\right)\nonumber\\
&=&\sum^{2^n}_{k=1}\mE^x\left(\mE^x\left(\frac{(k-1)u}{2^n}\leq \tau_D<\frac{ku}{2^n}, X^{t,x}_s\in A|\cF_{ku2^{-n}}\right)\right) \nonumber \\
&=&\sum^{2^n}_{k=1}\mE^x\left(\mI_{\frac{(k-1)u}{2^n}\leq \tau_D<\frac{ku}{2^n}}  \mE^x(\mI_{X^{t,x}_s\in A}|\cF_{ku2^{-n}})\right)\nonumber\\
&=&\sum^{2^n}_{k=1}\mE^x\left(\mI_{\frac{(k-1)u}{2^n}\leq \tau_D<\frac{ku}{2^n}}  \mE^{X^{t,x}_{ku2^{-n}}}(\mI_{X^{t,x}_{s-ku2^{-n}}\in A})\right)\nonumber\\
&=&\mE^x\left(  \mP^{X^{t,x}_{T_n}}(X^{t,x}_{s-T_n}\in A), \tau_D<u\right)\nonumber\\
&=&\mE^x\left(  \int_Ap(t,X^{t,x}_{T_n};s-T_n,y)dy, \tau_D<u\right).
\end{eqnarray}
If  $\tau_D<u$, then there is $n$ such that $s-T_n>s-u>0$.  Since $T_n\rightarrow \tau_D$  letting first $n\rightarrow \infty$,   and  then  $u\rightarrow s$ in (\ref{sec3-eq11}) yield
\ce
\mP^x(\tau_D<s, X^{t,x}_s\in A)&=&\int_A\mE^x( p(t,X^{t,x}_{\tau_D};s-\tau_D,y), \tau_D<s)dy\\
&=&\int_AK_D(t,x;s,y)dy.
\de
Thus
\ce
\mP^x(s<\tau_D, X^{t,x}_s\in A)&=&\int_Ap(t,x;s,y)dy-\mP^x(\tau_D< s, X^{t,x}_s\in A)\\
&=&\int_Ap_D(t,x;s,y)dy.
\de
Next,  we prove that $K_D(t,\cdot;s,\cdot)$ is jointly continuous and symmetric on $D\times D$. Suppose $(x_0,y_0)\in D\times D$ and $x_n$ converges to $x_0$ in $D$,  $y_n$ converges to $y_0$ in $D$. Let $\delta :=\min\{\dist(x_0, \partial D), \dist(y_0, \partial D)\}>0$. Choose $N$ large such that $\dist(y_n, \partial D)>\frac{1}{2}\delta $ for every $n>N$. 
By  (\ref{sec3-eq3a}), it is easy to see that for any $c>0$, $p(t,x;s,y)$ is bounded and uniformly continuous on the set
\[
D_c:= \{(s,x,y)\,;\ s>0, |x-y|>c, x,y \in \mR^d\}\,.
\]
 Set
\ce
M_c:=\sup_{(s,x,y)\in D_c }p(t,x;s,y).
\de
For any $0<r<s<t$, we have
\ce
 K_D(t,x_0;s,y_0)&=&  \mE^{x_0}( p(t,X^{t,x_0}_{\tau_D};s-\tau_D,y_0), \tau_D<s)\\
&=&\mE^{x_0}(  p(t,X^{t,x_0}_{\tau_D};s-\tau_D,y_0), \tau_D\leq r)\\
& &\qquad +\mE^{x_0}(  p(t,X^{t,x_0}_{\tau_D};s-\tau_D,y_0), r<\tau_D<s)\\
&=&\mE^{x_0}( p(t,X^{t,x_0}_{\tau_D};s-\tau_D,y_0), \tau_D\leq r)\\
& &\qquad  +  \mE^{x_0}( \mE^{X^{t,x_0}_r}[ p(t,X^{t,x_0}_{\tau_D};s-r-\tau_D,y_0), \tau_D<s-r], \tau_D>r ) \\
&=&\mE^{x_0}( p(t,X^{t,x_0}_{\tau_D};s-\tau_D,y_0), \tau_D\leq r)\\
& &\qquad  +\mE^{x_0}(\mE^{X^{t,x_0}_r}[ p(t,X^{t,x_0}_{\tau_D};s-r-\tau_D,y_0), \tau_D<s-r] )\\
& &\qquad  -\mE^{x_0}( \mE^{X^{t,x_0}_r}[\tau_D<s-r;p(t,X^{t,x_0}_{\tau_D};s-r-\tau_D,y_0)],
\tau_D\leq r).
\de
Thus
\begin{eqnarray}\label{sec3-eq011}
&&|K_D(t,x_0;s,y_0)-\mE^{x_0}\Big(\mE^{X^{t,x_0}_r}[ p(t,X^{t,x_0}_{\tau_D};s-r-\tau_D,y_0),
\tau_D<s-r]\Big)|
\nonumber \\
& &\qquad \le   \mE^{x_0}( p(t,X^{t,x_0}_{\tau_D};s-\tau_D,y_0), \tau_D\leq r) \nonumber\\
&&\qquad\qquad + \mE^{x_0}( \mE^{X^{t,x_0}_r}[ p(t,X^{t,x_0}_{\tau_D};s-r-\tau_D,y_0), \tau_D<s-r],\tau_D\leq r)
\nonumber\\
& &\qquad \le   2M_c\mP^{x_0}(\tau_D\leq r)\downarrow 0
\label{sec3-eq011}
\end{eqnarray}
as $r\downarrow 0$ uniformly in $x_0,y_0\in D\times D$. Set
\ce
\psi_r(x,y):=\mE^{x}(\mE^{X^{t,x}_r}[ p(t,X^{t,x_0}_{\tau_D};s-r-\tau_D,y), \tau_D<s-r]).
\de
For fixed $y$,   $\psi_r(\cdot,y)$ is continuous on $D$ by strong Feller property (Proposition \ref{prop3-2}),  while the family \{$\psi_r(x,\cdot)$\} is equi-continuous by the uniform continuity of $p(t,x;s,y)$ on $(s>0, |x-y|\geq \delta)$. It follows that $\psi_r(\cdot,\cdot)$ is jointly continuous on $D\times D$, so is $K_D(t,\cdot;s,\cdot)$
by (\ref{sec3-eq011}).

For any $s>r>0$, we have by (\ref{sec3-eq9}) and the Markov property of $X_{\cdot}^{t,x}$,
\ce
\int_Ap_D(t,x;s,y)dy&=&\mP^x(s<\tau_D, X^{t,x}_s\in A)\\
&=&\mE^x(r<\tau_D, \mP^{X^{t,x}_r}(s-r<\tau_D,X^{t,x}_{s-r}\in A))\\
&=&\mE^x(r<\tau_D, \int_Ap_D(r, X^{t,x}_r; s-r,y)dy)\\
&=&\int_D p_D(t,x;r,z)\int_Ap_D(r,z; s-r,y)dydz\\
&=&\int_A\int_Dp_D(t,x;r,z)p_D(r,z;s-r,y)dzdy
\de
for any $A\in \cB(D)$.   Hence (\ref{sec3-eq10}) follows by continuity of $p_D$.

To prove that $p_D$ is strictly positive, we first prove that for any $c>0$, if $0<t\leq\frac{c^2}{d}$ and $|x-y|<c<\dist(x,\partial D)\wedge\dist(y,\partial D)$, then $p_D(t,x;s,y)>0$. It is easy to verify that $(2\pi t)^{-d/2}\exp\{-c^2/2t\}$ is increasing in $t$ for $0<t\leq\frac{c^2}{d}$.  Hence,   for the above $t,x,y$, we have
\ce
K_D(t,x;s,y)\leq C(2\pi t)^{-d/2}\exp\{-c^2/2t\}
\de
and
\begin{eqnarray}\label{sec3-eq012}
p_D(t,x;s,y) &:=&p(t,x;s,y)-K_D(t,x;s,y)\nonumber \\
&\geq&   C(2\pi t)^{-d/2}[\exp\{-|x-y|^2/2t\}-\exp\{-c^2/2t\}]>0.
\end{eqnarray}
Now for any $t>0$ and $x, y\in D$, we connect $x$ and $y$ by a curve $\gamma$ in $D$ such that
\ce
\rho=\dist(\gamma, \partial D)>0.
\de
Then we can choose a sufficiently large integer $n$ such that $\frac{t}{n}<\frac{\rho^2}{4d}$ and such that
there exist points $a_0, a_1,\ldots ,a_{n+1}$ on $\gamma$ with $a_0=x, a_{n+l}=y$ and $a_i\in B(a_{i-1},\frac{\rho}{6})$, $i=1,2,\ldots, n+1$.   Then for any $x_i\in B(a_i,\frac{\rho}{6})$, we have $|x_i-x_{i+1}|\leq |x_i-a_{i}|+|a_{i+1}-x_{i+1}|+|a_i-a_{i+1}|<\frac{\rho}{2}$ and $\dist(x_i,\partial D)\geq\rho-\frac{\rho}{6}>\frac{\rho}{2}$. Therefore, by (\ref{sec3-eq10}) and (\ref{sec3-eq012}), we have
\ce
p_D(t,x;s,y)&=&\int_D\ldots\int_Dp_D(t,x;\frac{t}{n},x_1)p_D(t,x_1;\frac{t}{n},x_2)\ldots p_D(t,x_n;\frac{t}{n},y)dx_1\ldots dx_n\\
&\geq& \int_{B(a_1,\frac{\rho}{6})}\ldots\int_{B(a_n,\frac{\rho}{6})}p_D(t,x;\frac{t}{n},x_1)p_D(t,x_1;\frac{t}{n},x_2)  \\
&&\qquad \cdots p_D(t,x_n;\frac{t}{n},y)dx_1\ldots dx_n> 0.
\de
This completes the proof.
\end{proof}
This combined with \eqref{sec2.def.8} yields
\begin{corollary}\label{cor3-1}
Assume that the conditions (i)-(iii) in (({\bf H2.1}))  for the coefficients $a,b$ are satisfied in the Lipschitz domain  $D$.   Then the transition density $p_D( t,x; s,y) $ of the process $X^{t,x}_\cdot$ killed upon exiting the domain $D$
 satisfies that
\begin{equation}\label{sec3-eq3}
0<  p_D(  t,x; s,y) \leq C(s-t )^{-\frac{d}{2}}\exp\Big\{-\frac{c|x-y|^2}{2(s-t)}\Big\},
\end{equation}
for all $(s, y), ( t, x)\in[0, T ]\times D$ such that $t<s$,   and $C, c$ denote strictly positive constants depending on $d,T$ and the coefficients of the operator of $L_{s,y} $.
\end{corollary}


\Section{\bf Comparison  principle}
To study the stochastic partial differential equations
with the Laplacian replaced by
general elliptic operator \eqref{e.1.defL} it is useful to use the known results for the parabolic Anderson model \eqref{sec1-eq5} since there have been
a lot of profound results for the later equations. To do this, first we should obtain some results about
the Aronson type estimates for the solution corresponding to the stochastic differential equation
associated with the operator given by \eqref{e.1.defL}
completed  in Section \ref{s.2}  in our case).  Now we establish a lemma which can transfer some  results
obtained for parabolic Anderson model to general elliptic operator case.

\begin{lemma}\label{l.3.1} Let
\[
\begin{split}
\cC_n
 =&\left\{ F=f(X_{t_1}, \cdots, X_{t_n})\,;
\ f: \RR^{nd}\to \RR \ \hbox{such that    } \right.\\
&\qquad \  \hbox{$\int_{D} |f(y_1, \cdots, y_n)|dy_1\dots dy_n$ is finite in any bounded domain  } \\
	&\qquad \  \left. \hbox{$D$ of $\RR^{nd}$ and is of   polynomial growth}\,, 0< t_1, \cdots, t_n<\infty\right\}
\end{split}
\]
be the set of   cylindrical functionals  of the
diffusion process $X_\cdot$, for which the  Aronson estimates holds true:
\begin{equation}
\PP\left[ X_t\in dy\big| X_s=x\right]=
p(s,x;t,y)dy\,,\quad x,y\in \RR^d,\quad  0<s<t<\infty
\end{equation}
for some $p( s,x;t,y)$ satisfying
\begin{equation}
	p(s,x;t,y)\le \kappa_1  (2\pi \kappa_2 (t-s))^{-d/2}
	\exp\left[ -\frac{|y-x|^2}{2\kappa_2  (t-s)}\right]\,, \quad \forall \ x,y\in\RR^d,  s<t \,,
	\label{e.3.2}
\end{equation}
where $\kappa_1, \kappa_2$ are two positive constants.   Then for any element $F=f(X_{t_1}, \cdots, X_{t_n})$ in $\cC_n$
with $f\ge 0$, we have
\begin{equation}
	\EE \left[  f(X_{t_1}, \cdots, X_{t_n})\right]\le \kappa_1^n \EE \left[  f(B_{\kappa _2 t_1}, \cdots,  B_{ \kappa _2 t_n})\right]\label{e.3.3}
\end{equation}
\end{lemma}

\begin{proof}
	Denote the heat kernel  $q_t(x)=(2\pi t)^{-d/2} e^{-\frac{|x|^2}{2t}}$.
	We can arrange the time $t_1, \cdots, t_n$ in an increasing order,
	say we can assume  $t_1<\cdots<t_n$. Then
by Markov property of the process $X_t$,
and then by  \eqref{e.3.2}  we have
\bas
&&\EE \left[  f(X_{t_1}, \cdots, X_{t_n})\right]
 = 	\EE\left\{ \EE \left[  f(X_{t_1}, \cdots, X_{t_n})\big|\cF_{t_{n-1}}\right]\right\}\\
&&\qquad\qquad =  \int_{\RR^d} p(t_{n-1}, X_{t_{n-1}};t_n, y_n)
f(X_{t_1}, \cdots, X_{t_{n-1}}, y_n) ) dy_n=\cdots\\
&&\qquad\qquad =  \int_{\RR^{nd}} p(0,x;t_1,y_1)\cdots p(t_{n-2}, y_{n-2}; t_{n-1}, y_{n-1}) \\
&&\qquad\qquad\qquad \times p(t_{n-1}, y_{n-1};t_n, y_n)
f(y_1, \cdots, y_{n-1}, y_n) ) dy_1\cdots dy_{n-1} dy_n \\
&&\qquad\qquad \le \kappa_1^n   \int_{\RR^{nd}} q_{\kappa_2 t_1}(y_1-x) \cdots q_{\kappa_2 (t_{n-1}-t_{n-2})}(y_{n-1}-y_{n-2}) \\
&&\qquad\qquad\qquad \times q_{\kappa_2 (t_{n }-t_{n-1})}(y_{n}-y_{n-1})
f(y_1, \cdots, y_{n-1}, y_n) ) dy_1\cdots dy_{n-1} dy_n \,.
\eas
It is obviously that the last expression can be written as $\kappa_1^n \EE \left[ f(B_{\kappa_2 t_1}, \cdots, B_{\kappa_2 t_{n-1}},
B_{\kappa_2t_n})\right]$.
	This proves the lemma.
\end{proof}

\begin{remark}   The conditions that $\int_{D} |f(y_1, \cdots, y_n)|dy_1\dots dy_n$ is finite in any bounded domain  $D$ of $\RR^{dn}$ and is of   polynomial growth is to guarantee that the left hand side of \eqref{e.3.3} is finite. If the
	left hand side of \eqref{e.3.3} is infinite,  then obviously, the right hand side of
	\eqref{e.3.3} is also infinite.  If we allow infinite then we can remove the
	condition that $\int_{D} |f(y_1, \cdots, y_n)|dy_1\dots dy_n$ is finite in any bounded domain  $D$ of $\RR^{dn}$ and is of   polynomial growth.
\end{remark}

A straightforward consequence of the above lemma is the following.
\begin{lemma}\label{l.3.3}
	Let the diffusion process $X_t$ satisfies the Aronson estimate as in Lemma \ref{l.3.1}.
	Let $f, g$ be two nonnegative functions, then there exists a positive constant $C$ depending on $d,T$ and the coefficients of operator $L_t$ such that
	\begin{equation}\label{sec6.2-eq06}
		\mE \exp\left\{\int_0^t\int_0^tf(s,r)g(X^{t,x}_s,X^{t,x}_r)dsdr\right\}\leq C_1\mE \exp\left\{c_2 \int_0^{t }\int_0 ^{t } f(s,r)g(B_{\kappa_2  s}, B_{\kappa_2 r} )dsdr\right\},
	\end{equation}
	where $X^{t,x}_{\cdot}$ is the solution of equation (\ref{sec2-eq01}) having  the Aronson bound \eqref{e.3.2}.
\end{lemma}

\section{\bf Approximation of the Gaussian noise and Feynman-Kac functional}
\subsection{Noise structure and Malliavin calculus}
Let $L_t$ be defined by \eqref{e.1.defL} with the coefficients satisfying the hypothesis ({\bf H2.1}).   We consider the following stochastic partial differential equation  in a bounded Lipschitz domain $D$ of  $\mR^d$:
\begin{equation}\label{sec2-eq2}
\left\{
                  \begin{array}{lll}
                 \frac{\partial u(t,x)}{\partial t}=L_tu(t,x)+u(t,x)\dot{W}(t,x),~~~x\in D,\\
                  u(0,x)=f(x), ~~~x\in D,\\
                  u(t,x)=g(t,x),~~~(t,x)\in(0,T)\times\partial D.
                  \end{array}
                    \right.
\end{equation}
Let  $ {W}(t,x)$ be  fractional Brownian field  in space and time, i.e.
\ce
\mE(W(t,x)W(s,y))=R_{H_0}(s,t)\prod^d_{i=1}R_{H_i}(x_i,y_i),
\de
where for any $H_i\in(0,1)$, $i=0,1,\ldots, d$,   $R_{H_i}(s,t)$  denotes
the covariance function of the fractional Brownian motion with Hurst parameter $H_i$, that is
\ce
R_{H_i}(s,t)=\frac{1}{2}(|t|^{2H_i}+|s|^{2H_i}-|t-s|^{2H_i}), ~~~i=0,1,\ldots, d.
\de
$\dot{W}(t,x)=\frac{\partial ^{d+1}}{\partial t, \partial x_1 \cdots \partial x_d} W(t,x)$.
Thus, the covariance of  $\dot{W}(t,x)$ is
\ce
\mE(\dot W(t,x) \dot W(s,y))=\kappa_{H_0}(s,t)\prod^d_{i=1}\kappa_{H_i}(x_i,y_i),
\de
where
\ce
\kappa_{H_i}(s,t)=H_i(2H_i-1) |t-s|^{2H_i-2} , ~~~i=0,1,\ldots, d.
\de
Let $\cF_t=\sigma(W(A);A\in\cB((0,t)\times D))\vee\cN$ where $\cN$ is the class of $\mP$-null sets in $\cF$, and $\cP$ denote the $\sigma$-algebra of $\cF_t$ progressively measurable subsets of $\Omega\times(0,T)$.

Denote by $\cE$ the linear span of the indicator functions of rectangles of the form $(s,t]\times(x,y]$ in $\mR_+\times\mR^d$. Consider, in $\cE$, the inner product defined by
\ce
\langle I_{(0,s]\times(0,x]},I_{(0,t]\times(0,y]}\rangle_{\cH}=R_{H_0}(s,t)\prod^d_{i=1}R_{H_i}(x_i,y_i).
\de
We denote by $\cH$ the closure of $\cE$ with respect to this inner product. The mapping $W:I_{(0,t]\times(0,x]}\rightarrow W(t,x)$ extends to a linear isometry between $\cH$ and the Gaussian space spanned by $W$. We will denote this isometry by
\ce
W(\phi):=\int_0^{\infty}\int_{\mR^d}\phi(t,x)W(dt,dx)
\de
if $\phi\in\cH$. Notice that if $\phi$ and $\psi$ are functions in $\cE$, then
\begin{equation}\label{sec2-eq3}
\begin{aligned}
\mE(W(\phi)W(\psi))=\langle \phi,\psi\rangle_{\cH}
&=& \int_{\mR_+^2}\int_{\mR^{2d}}\phi(s,x)\psi(t,y)\kappa_{H_0}(s,t)\\
&&\qquad \times\prod^d_{i=1}\kappa_{H_i}(x_i, y_i)dxdydsdt\,.
\end{aligned}
\end{equation}
Furthermore, $\cH$ contains the class of (generalized)  functions $\phi$ on $\mR_+\times\mR^d$ such that
\begin{equation}\label{sec2-eq4}
\int_{\mR_+^2}\int_{\mR^{2d}} \phi(s,x)\phi(t,y)  |s-t|^{2H_0-2}\times\prod^d_{i=1}|x_i-y_i|^{2H_i-2}dxdydsdt<\infty.
\end{equation}

We denote by $C^\infty_p(\mR^m)$ the set of all   smooth  functions $f: \mR^m\rightarrow\mR$ such that $f$ and all of its partial derivatives have polynomial growth.

Let $\cS$ denote the class of smooth random variables   $F\in\cS$
of  the form
\begin{equation}\label{sec2-eq5}
F=f(W(h_1),\ldots,W(h_m)),
\end{equation}
where $f$ belongs to $C^\infty_p(\mR^m)$ and $h_1,\ldots,h_m\in\cH$ and $m \geq1$.

We will make use of the notation $\partial_if=\frac{\partial f}{\partial x_i}$,  $\partial_{ij}f=\frac{\partial^2 f}{\partial x_i\partial x_j}$ and $\nabla f=(\partial_if, \ldots, \partial_mf)$, whenever $f\in C^1(\mR^m)$.

The Malliavin derivative of a smooth random variable $F$ of the form (\ref{sec2-eq5}) is the $\cH$-valued random variable given by
\ce
\mD F=\sum^m_{i=1}\partial_if(W(h_1),\ldots,W(h_m))h_i.
\de
In particular, $\mD W(h)=h$ for any element $h\in\cH$.

We  interpret $\mD F:\Om\to \cH$ as a directional derivative so that for any element $h\in\cH$ we have
\ce
&&\langle \mD F, h\rangle_{\cH}\\
&=&\lim_{\varepsilon\rightarrow0}\frac{1}{\varepsilon}\Big(f(W(h_1)+\varepsilon\langle h_1,h\rangle_{\cH},\ldots,W(h_m)+\varepsilon\langle h_m,h\rangle_{\cH})-f(W(h_1),\ldots,W(h_m))\Big).
\de
The operator $\mD$ is closable from $L^p(\Omega)$ to $L^p(\Omega;\cH)$ for any $p\geq1$ (Proposition 1.2.1 in \cite{Nualart}).

For any $p\geq1$, we will denote the domain of $\mD$ in $L^p(\Omega)$ by $\mD^{1,p}$, meaning that $\mD^{1,p}$ is the closure of the class of smooth random variables $\cS$ with respect to the norm
\ce
||F||_{1,p}=\left(\mE(|F|^p)+\mE(||\mD F||^p_{\cH})\right)^{1/p}.
\de
We denote by $\delta$ the adjoint of the operator $\mD$. That is, $\delta(\cdot)$ is an unbounded operator on $L^2(\Omega;\cH)$ with values in $L^2(\Omega)$.  We say $u\in L^2(\Omega)$ is in the domain $\Dom(\delta)$  if
there is a $\delta(u)$ such that
\begin{equation}\label{sec2-eq6}
\mE(F\delta(u))=\mE(\langle\mD F,u\rangle_{\cH}),
\end{equation}
for  any $F\in\mD^{1,2}$. The operator $\delta(\cdot)$ is called the divergence operator and is closed as the adjoint of an unbounded and densely defined operator. We will interpret the divergence operator $\delta(\cdot)$ as a stochastic integral and we will call it the Skorohod integral because in the Brownian motion case it coincides with the generalization of the It\^o stochastic integral to anticipating integrands. 

The space $\mD^{1,2}(\cH)$ is included in the domain of $\delta(\cdot)$. If $u,v\in\mD^{1,2}(\cH)$, then
\begin{equation}\label{sec2-eq7}
\mE\Big(\delta(u)\delta(v)\Big)=\mE\Big(\langle u,v\rangle_{\cH}\Big)+\mE\Big(\Tr(\mD u\circ\mD v)\Big)
\end{equation}
(Proposition 1.3.1 in \cite{Nualart}).

Let $h\in\cH$ and $F\in\mD^{h,2}$. Then $Fh$ belongs to the domain of $\delta(\cdot)$ and the following equality is true
\begin{equation}\label{sec2-eq8}
FW(h)=\delta(Fh)+\langle\mD F,h\rangle_{\cH}
\end{equation}
(Proposition 1.3.4 in \cite{Nualart}).

\subsection{Regularization of the noise and Feynman-Kac functional}
If $\dot W(t,x)$ in \eqref{sec1-eq4}  is a bounded continuous function,
then   the solution to that equation is formally
given by the Feynman-Kac formula
\eqref{sec1-eq3}  with $c(t,x)$ replaced by $\dot W(t,x)$. This will produce a term     $\int_0^{\tau_D}\dot W(t-s,X_s^{t,x})ds$,  where  $\{X_s^{t,x}\,, s\ge 0\}$ is the weak solution to
\eqref{sec1-eq2}.   So,
the   first thing we need to do is to give   a meaning to  this expression.
Formally we have
\[
\dot W(t,x)=\int_0^t \int_{  \RR^d} \delta(t-r)\delta( x-y) \dot W(r,y) drdy=\int_0^t \int_{  \RR^d}
\delta(t-r, x-y) d W(r,y)
\,.
\]
Thus heuristically, we have
\[
\begin{split}
 \int_0^{\tau_D}\dot W(t-s,X_s^{t,x})ds
 & =\int_0^{\tau_D}\int_0^t
\int_{  \RR^d} \delta(t- s-r )\delta( X_s^{t,x}-y) d W(r,y)   ds\\
 & =\int_0^t\int_{  \RR^d} \int_0^ {\tau_D}
\delta(t-s-r )\delta( X_s^{t,x}-y) ds d W(r,y)\\
& =\int_0^t\int_{  \RR^d} \int_{-\infty} ^\infty \mI_{[0, \tau_D]}(s)
\delta(t-s-r )\delta( X_s^{t,x}-y) ds d W(r,y)\\
& =\int_0^t\int_{  \RR^d}  \mI_{[0, \tau_D]}(t-r)
 \delta( X_{t-r}  ^{t,x}-y)   d W(r,y)\\
 & = \int_{  \RR^d}  \int_{t-\tau_D}^t
 \delta( X_{t-s}  ^{t,x}-y)   d W(s,y)\\
 & = \int_0^{\tau_D} \int_{   \RR^d}
 \delta( X_s  ^{t,x}-y)     W(t-ds,dy)
\,.
\end{split}
\]
To define the stochastic integral for Dirac integrands,
we approximate the Dirac function by heat kernel approximation:
For any $\varepsilon>0$, we
denote by $p_\varepsilon(z)$ the $d$-dimensional heat kernel:
\ce
p_\varepsilon(z)=\frac{1}{(\sqrt{2\pi\varepsilon})^d}
e^{-\frac{|z|^2}{2\varepsilon}},~~~z\in\mR^d.
\de
On the other hand, to make our argument rigorous, we shall approximate the noise $W$ by a differentiable one.  For this reason we smooth the spatial variable of the noise by using the above approximation $p_\vare(z)$ and
since the time variable takes only positive values we also introduce another approximation of the Dirac delta function for the time variable.  For any $\delta>0$, we define the rectangular kernel:
\ce
\varphi_{\delta}(t)=\frac{1}{\delta}\mI_{[0,\delta]}(t).
\de
Then $\varphi_{\delta}(t)p_\varepsilon(z)$ provides an approximation of the Dirac delta function $\delta (t)\delta (z)$ as $\varepsilon$ and $\delta$ tend to zero.
We denote by $W^{\varepsilon,\delta}$ the approximation of the fractional Brownian sheet $W(t,z)$ defined by
\begin{equation}\label{sec4-eq1}
W^{\varepsilon,\delta}(t,z):=\int_0^{\infty}\int_{\mR^d}\varphi_\delta(t-s)p_\varepsilon(z-y)W(s,y)dyds.
\end{equation}
We can write
\[
\int_0^{t\wedge\tau_D}  \dot W^{\varepsilon,\delta}(t-r,X_r^{t,x})dr
=
\int_0^{\infty}\int_{\mR^d}A_{t\wedge\tau_D,x}^{\varepsilon,\delta}(r,y)W(dy,dr)
\]
where and throughout the remaining part of this paper 
\begin{equation}\label{sec4-eq2}
A_{t\wedge\tau_D,x}^{\varepsilon,\delta}(r,y):=\int_0^{t\wedge\tau_D}\varphi_\delta(t-r-s)p_\varepsilon(X^{t,x}_s-y)ds,
\end{equation}
and  $X^{t,x}_{\cdot}$ is the solution
of Equation (\ref{sec1-eq2}) under Assumption ({\bf H2.1}).

First, we will show that for any $\varepsilon>0$ and $\delta>0$, the function $A_{t\wedge\tau_D,x}^{\varepsilon,\delta}(r,y)$ belongs to the space $\cH$ almost surely and the family of random variables
\begin{equation}\label{sec4-eq3}
V_{t,\tau_D,x}^{\varepsilon,\delta}:=\int_0^{\infty}\int_{\mR^d}A_{t\wedge\tau_D,x}^{\varepsilon,\delta}(r,y)W(dy,dr)=W(A_{t\wedge\tau_D,x}^{\varepsilon,\delta})
\end{equation}
converges in $L^2(\Omega)$ as $\varepsilon$ and $\delta$ tend to zero.
We write    $X^{t,x}
=(X^{1,t,x}, \cdots, X^{d,t,x})^{'}$.

\begin{theorem}\label{th4-1} Suppose that ({\bf H2.1}) holds  true and $2H_0+\sum^d_{i=1}H_i-d-1>0$. Then, for any $\varepsilon>0$, and $\delta>0$, $A_{t\wedge\tau_D,x}^{\varepsilon,\delta}$
defined by  (\ref{sec4-eq2}) belongs to $\cH$ and the family of random variables $V_{t,\tau_D,x}^{\varepsilon,\delta}$ defined in (\ref{sec4-eq3}) converges in $L^2(\Omega)$ to a limit denoted by
\begin{equation}\label{sec4-eq4}
V_{t,\tau_D,x}:=\int_0^{t\wedge\tau_D}\int_{\mR^d}\delta (X^{t,x}_r-y)W(dy,dr).
\end{equation}
Conditional on the Brownian motion  $B$, $V_{t,\tau_D,x}$ is a Gaussian random variable with mean 0 and with variance given by
\begin{equation}\label{sec4-eq5}
Var^W(V_{t,\tau_D,x})=\alpha_H\int_0^{t\wedge\tau_D}\int_0^{t\wedge\tau_D}|r-s|^{2H_0-2}\prod^d_{i=1}|X^{i,t,x}_r-X^{i,t,x}_s|^{2H_i-2}drds.
\end{equation}
\end{theorem}

\begin{proof} Fix $\varepsilon,\varepsilon',\delta,\delta'>0$. First let us compute the inner product
\begin{eqnarray}  \langle A_{t\wedge\tau_D,x}^{\varepsilon,\delta},A_{t\wedge\tau_D,x}^{\varepsilon',\delta'}\rangle_{\cH}&=&\alpha_H\int_{\mR^2_+}\int_{\mR^{2d}}
A_{t\wedge\tau_D,x}^{\varepsilon,\delta}(r,y)A_{t\wedge\tau_D,x}^{\varepsilon',\delta'}(s,z) \nonumber\\
&&\qquad \times |r-s|^{2H_0-2}\times\prod^d_{i=1}|y_i-z_i|^{2H_i-2}dydzdrds \nonumber\\
&=&\alpha_H\int_{\mR^2_+}\int_{[0,t\wedge\tau_D]^2}\int_{\mR^{2d}}p_\varepsilon(X^{t,x}_u-y)p_{\varepsilon'}(X^{t,x}_v-z) \nonumber \\
&&\qquad\times \varphi_\delta(t-r-u)\varphi_{\delta'}(t-s-v) \nonumber \\
&&\times |r-s|^{2H_0-2}\times\prod^d_{i=1}|y_i-z_i|^{2H_i-2}dydzdudvdrds.\nonumber \\ \label{sec4-eq6}
\end{eqnarray}
By Lemmas A.2 and A.3 in \cite{Hu}, we have the estimate
\begin{eqnarray}\label{sec4-eq7}
&&\int_{\mR^2_+}\int_{\mR^{2d}}p_\varepsilon(X^{t,x}_u-y)p_{\varepsilon'}(X^{t,x}_v-z)
\varphi_\delta(t-r-u)\varphi_{\delta'}(t-s-v) \nonumber\\
&&\qquad \qquad \times |r-s|^{2H_0-2}\times\prod^d_{i=1}|y_i-z_i|^{2H_i-2}dydzdrds   \nonumber\\
&&\qquad \leq C|u-v|^{2H_0-2}\times\prod^d_{i=1}|X^{i,t,x}_u-X^{i,t,x}_v|^{2H_i-2}
\end{eqnarray}
for some constant $C>0$. 

Since the transition probability density of $X^{ t,x}_s$
satisfies  the Aronson type estimate \eqref{sec3-eq3}, the condition of    Lemma \ref{l.3.1} is met. Thus, we have
\begin{eqnarray}\label{sec4-eq8}
&& \int_0^{t}\int_0^{t}|s-r|^{2H_0-2}\times
\mE^B\left( \prod^d_{i=1}|X^{i,t,x}_s-X^{i,t,x}_r|^{2H_i-2}dsdr\right) \nonumber\\
&& \qquad \le C\int_0^t\int_0^r|s-r|^{2H_0-2}\mE^B\left(\prod^d_{i=1}|B^i_s-B^i_r|^{2H_i-2}\right)dsdr \nonumber\\
&&\qquad \leq C_{H_i}\int_0^t\int_0^r(r-s)^{2H_0-2}(r-s)^{\sum^d_{i=1}H_i-d}dsdr \nonumber\\
&&\qquad \leq C_{H_i}t^{2H_0+\sum^d_{i=1}H_i-d}<\infty
\end{eqnarray}
since
\ce
2H_0+\sum^d_{i=1}H_i-d-1>0.
\de
As a consequence, taking the mathematical expectation with respect to $B$ in (\ref{sec4-eq6}), letting $\varepsilon=\varepsilon'$ and $\delta=\delta'$ and using the estimates (\ref{sec4-eq7}) and (\ref{sec4-eq8}) yield
\ce
\mE^B\|A_{t\wedge\tau_D,x}^{\varepsilon,\delta}\|^2_{\cH}\leq C<\infty.
\de
This implies that almost surely $A_{t\wedge\tau_D,x}^{\varepsilon,\delta}$ belongs to the space $\cH$ for all $\varepsilon$ and $\delta>0$. Therefore, the random variables $V_{t,\tau_D,x}^{\varepsilon,\delta}=W(A_{t,x}^{\varepsilon,\delta})$ are well defined and we have
\ce
\mE^B\mE^W(V_{t,\tau_D,x}^{\varepsilon,\delta}V_{t,\tau_D,x}^{\varepsilon',\delta'})=\mE^B\langle A_{t\wedge\tau_D,x}^{\varepsilon,\delta},A_{t\wedge\tau_D,x}^{\varepsilon',\delta'}\rangle_{\cH}.
\de
For any $s\neq r$ and $X^{t,x}_r\neq X^{t,x}_s$, as $\varepsilon,\delta, \varepsilon',\delta'$ tend to zero, the left-hand side of the inequality (\ref{sec4-eq7}) converges to $|s-r|^{2H_0-2}\prod^d_{i=1}|X^{i,t,x}_s-X^{i,t,x}_r|^{2H_i-2}$ almost surely. Therefore, by the dominated convergence theorem
(due to the bounds \eqref{sec4-eq7}-\eqref{sec4-eq8}), we obtain that
\ce
&&\lim_{\varepsilon,\delta, \varepsilon',\delta'\rightarrow 0}\mE^B\mE^W(V_{t,\tau_D,x}^{\varepsilon,\delta}V_{t,\tau_D,x}^{\varepsilon',\delta'})\\
&&=\alpha_{H}\int_0^t\int_0^t|s-r|^{2H_0-2}\mE^B\left(\mI_{\{s,r\leq\tau_D\}}\times\prod^d_{i=1}|X^{i,t,x}_s-X^{i,t,x}_r|^{2H_i-2}\right)dsdr\\
&&\leq \alpha_{H}\int_0^t\int_0^t|s-r|^{2H_0-2}\mE^B\left(\prod^d_{i=1}|X^{i,t,x}_s-X^{i,t,x}_r|^{2H_i-2}\right)dsdr<\infty.
\de
We can easily show
\bas
&&\lim_{\varepsilon,\delta, \varepsilon',\delta'\rightarrow 0}\mE(V_{t,\tau_D,x}^{\varepsilon,\delta}-V_{t,\tau_D,x}^{\varepsilon',\delta'})^2\\
&&\qquad\qquad =\mE(V_{t,\tau_D,x}^{\varepsilon,\delta})^2-2\mE(V_{t,\tau_D,x}^{\varepsilon,\delta}V_{t,\tau_D,x}^{\varepsilon',\delta'})+\mE(V_{t,\tau_D,x}^{\varepsilon',\delta'})^2\rightarrow 0\,.
\eas
This implies that $V_{t,\tau_D,x}^{\varepsilon,\delta}$ is a Cauchy sequence in $L^2(\Omega)$ for all sequences $\varepsilon$ and $\delta$ converging to zero. As a consequence, $V_{t,\tau_D,x}^{\varepsilon,\delta}$ converges in $L^2(\Omega)$ to a limit denoted by $V_{t,\tau_D,x}$ in (\ref{sec4-eq4}) which does not depend on the choice of the sequences $\varepsilon$ and $\delta$.

Finally,  by a similar argument, we show (\ref{sec4-eq5}).
\end{proof}

The next result provides the exponential integrability of the random variable $V_{t,\tau_D,x}$ defined in (\ref{sec4-eq4}).
\begin{proposition}\label{prop4-2} Suppose that Assumption ({\bf H2.1}) holds true and $2H_0+\sum^d_{i=1}H_i-d-1>0$. Then for any $\lambda\in\mR$, we have
\begin{equation}
\begin{split}
&\mE\exp\left\{\lambda\int_0^{t\wedge\tau_D}\int_{\mR^d}\delta (X^{t,x}_r-y)W(dr,dy)\right\}\\
&\qquad =\EE \left[ \exp \left\{ \alpha_H\lambda^2 \int_0^{t\wedge\tau_D}\int_0^{t\wedge\tau_D}|r-s|^{2H_0-2}\prod^d_{i=1}|X^{i,t,x}_r-X^{i,t,x}_s|^{2H_i-2}drds\right\}\right]<\infty.
\end{split} \label{sec4-eq9}
\end{equation}
\end{proposition}

\begin{proof}
From (\ref{sec4-eq5})  it follows
\cen
\mE\exp\{\lambda V_{t,\tau_D,x}\}&=&\mE^B\exp\Big\{\frac{\lambda^2}{2}\alpha_H \int_0^{t\wedge\tau_D}\int_0^{t\wedge\tau_D}|s-r|^{2H_0-2}\nonumber \\
&&\qquad \times \prod^d_{i=1}|X^{i,t,x}_s-X^{i,t,x}_r|^{2H_i-2}dsdr\Big\}\nonumber\\
&=&\mE^B\exp\left\{\frac{\lambda^2}{2}\alpha_H Y_{t\wedge\tau_D,x}\right\}\nonumber \\
&=&\sum^{\infty}_{n=0}\frac{1}{n!2^n}(\lambda^2\alpha_H)^n\mE^B(Y_{t\wedge\tau_D,x})^n\,, \label{e.3.17}
\den
where
\ce
Y_{t\wedge\tau_D,x}:=\int_0^{t\wedge\tau_D}\int_0^{t\wedge{\tau_D}}|s-r|^{2H_0-2}\times\prod^d_{i=1}|X^{i,t,x}_s-X^{i,t,x}_r|^{2H_i-2}dsdr.
\de
First we have
\[
\mE^B(Y_{t\wedge\tau_D,x})^n
\le   \mE^B\left[ \int_0^{t }\int_0^{t }|s-r|^{2H_0-2}\times\prod^d_{i=1}|X^{i,t,x}_s-X^{i,t,x}_r|^{2H_i-2}dsdr\right]^n \,.
\]
By Lemma \ref{l.3.1}, we have
\bas
\mE^B(Y_{t\wedge\tau_D,x})^n
&\le&  C^n \mE^B\left[ \int_0^{t }\int_0^{t }|s-r|^{2H_0-2}\times\prod^d_{i=1}|B^i  _s-B^i_r|^{2H_i-2}dsdr\right]^n\,.
\eas
With these comparison inequality,  our theorem is a consequence of \cite[Theorem 3.3]{Hu}.
\end{proof}

\section{\bf Feynman-Kac formula}
We recall the approximation
\eqref{sec4-eq1} of the   fractional Brownian field
and give  the following definition of the Stratonovich integral.

\begin{definition}\label{df5-1}
Given a random field $u=\{u(t,x), t\geq0,x\in D\}$ such that
\ce
\int_0^T\int_{D}|u(t,x)|dxdt<\infty
\de
almost surely for all $T>0$, the Stratonovich integral $\int_0^T\int_{D}u(t,x)W(dx,dt)$ is defined as the following limit in probability, if it exists:
\ce
\lim_{\varepsilon,\delta\downarrow0}\int_0^T\int_{D}u(t,x)\dot{W}^{\varepsilon,\delta}(t,x)dxdt.
\de
\end{definition}

\begin{definition}\label{df5-2} An adapted random field $u=\{u(s,y), s\geq0,y\in D\}$ such that $\mE(u^2(s,y))<\infty$ for all $(s,y)$ is a mild solution of (\ref{sec1-eq4}) if, for any $(t,x)\in [0,T]\times D$, the process $\{p_D(t,x;s,y)u(s,y)\mI_{s\in[0,t]}, y\in D\}$ is Stratonovich  integrable, and the following equation holds
\ba
u(s,y)&=&\int_{D}f(x)p_D(t,x;0,y)dx+\int_0^s\int_Dp_D(s,y;r,z)u(r,z)W(dz,dr)\nonumber \\
& &\qquad +\int_0^s\int_{\partial D}p_D(s,y;r,z)g(r,z)d{\bf S}(z)dr\,, \label{e_mild_sol}
\ea
where $p_D$ denotes the transition probability density kernel  of the solution to equation (\ref{sec3-eq4}) on $D$, $d{\bf S}$ is the surface measure of the boundary $\partial D$ and the second term is a Stratonovich stochastic integral in the sense of Definition \ref{df5-1}.
\end{definition}

The following theorem is the main result of this section.

\begin{theorem}\label{th5-1} Suppose that ({\bf H2.1}) hold true for the coefficients $b$ and $\si$ and $2H_0+\sum^d_{i=1}H_i-d-1>0$ and that $f, g$ are continuous bounded measurable functions. Let $X_\cdot^{t,x}$ satisfy \eqref{sec1-eq2}.
Then, the random field
\begin{eqnarray}
u(t,x)&=&\mE^B\Big(f(X^{t,x}_t)\mI_{t\le \tau_D}\exp\Big\{\int_0^t\int_{\mR^d}\delta (X_s^{t,x}-y)W(dy,ds)\Big\} \nonumber \\
& &\quad   +g(\tau_D,X^{t,x}_{\tau_D})\mI_{t>\tau_D}\exp\Big\{\int_0^{\tau_D}\int_{\mR^d}\delta (X_s^{t,x}-y)W(dy,ds)\Big\}\Big)\nonumber\\
&=& \mE^B\Big(h(t\wedge \tau_D , X^{t,x}_{t
\wedge \tau_D}) \exp\Big\{\int_0^{t\wedge \tau_D} \int_{\mR^d}\delta (X_s^{t,x}-y)W(dy,ds)\Big\}
 \label{sec5-eq2}
\end{eqnarray}
is a mild solution of (\ref{sec1-eq4}),
where
\begin{equation}
h(t,x)=\begin{cases}
f(x)  & \qquad \hbox{when}\  t\ge 0\,, x\in D\\
g(t, x) & \qquad \hbox{when}\  t\ge 0\,, x\in \partial D\,.\\
\end{cases}
\end{equation}
\end{theorem}

\begin{proof}
Consider the approximation of (\ref{sec1-eq4}) given by the following Cauchy-Dirichlet parabolic equation with a random potential:
\begin{equation}\label{sec5-eq3}
\left\{
                  \begin{array}{lll}
                 \frac{\partial u^{\varepsilon,\delta}(t,x)}{\partial t}=L_tu^{\varepsilon,\delta}(t,x)+u^{\varepsilon,\delta}(t,x)\dot{W}^{\varepsilon,\delta}(t,x),~~~0<t<\infty ,~~x\in D,\\
                  u^{\varepsilon,\delta}(0,x)=f(x), ~~~x\in D,\\
                  u^{\varepsilon,\delta}(t,x)=g(t,x),~~~(t,x)\in(0,\infty)\times\partial D.
                  \end{array}
                    \right.
\end{equation}
Since now for any fixed $\varepsilon,\delta$,  $\dot{W}^{\varepsilon,\delta}(t,x)$
is a continuous function on $[0, \infty)\times D$,  the solution to the above equation
exists uniquely and is given by the Feynman-Kac formula.

It is easy to see  that $h(t,x)$ is a
continuous  function   on the boundary of the cylinder $U=\{x\in D,0<t<\infty \}$ coinciding with $f(x)$ on the base of the cylinder and with $g(t,x)$ on the lateral surface.

From the classical Feynman-Kac formula   \cite[Theorem 2.3]{Freidlin}, we have that
\begin{eqnarray}
u^{\varepsilon,\delta}(t,x)&=&\mE^B\Big(f(X^{t,x}_t)\mI_{t\le \tau_D}\exp\left\{\int_0^t\dot{W}^{\varepsilon,\delta}(t-s,X_s^{t,x})ds\right\} \nonumber \\
&&\qquad+g(\tau_D,X^{t,x}_{\tau_D})\mI_{t>\tau_D}\exp\left\{\int_0^{\tau_D}\dot{W}^{\varepsilon,\delta}(t-s,X_s^{t,x})ds\right\}\Big)\nonumber \\
&=&\mE^B\left(h(t\wedge\tau_D,X^{t,x}_{t\wedge\tau_D})\exp\left\{\int_0^{t\wedge\tau_D}\dot{W}^{\varepsilon,\delta}(t-s,X_s^{t,x})ds\right\}\right),\label{sec5-eq4}
\end{eqnarray}
where $X^{t,x}_s$ is a solution starting at $x$
($X^{t,x}_0=x$)  of stochastic differential equation (\ref{sec2-eq01}) driven by $d$-dimensional Brownian motion   independent of $W$. By (\ref{sec4-eq1}) and Fubini's theorem, we can write
\ce
\int_0^{t\wedge\tau_D}\dot{W}^{\varepsilon,\delta}(t-s,X_s^{t,x})ds&=&\int_0^{t\wedge\tau_D}\int_0^{\infty}\int_{\mR^d}\varphi_\delta(t-r-s)p_\varepsilon(X_s^{t,x}-y)W(dy,dr)ds\\
&=&\int_0^{\infty}\int_{\mR^d}\int_0^{t\wedge\tau_D}\varphi_\delta(t-r-s)p_\varepsilon(X_s^{t,x}-y)dsW(dy,dr)\\
&=&\int_0^{\infty}\int_{\mR^d}A_{t\wedge\tau_D,x}^{\varepsilon,\delta}(r,y)W(dy,dr)\\
&=&V_{t,\tau_D,x}^{\varepsilon,\delta}\,,
\de
where $A_{t\wedge\tau_D,x}^{\varepsilon,\delta}(r,y)$ is defined by \eqref{sec4-eq2}
and $V_{t,\tau_D,x}^{\varepsilon,\delta}$ is defined by \eqref{sec4-eq3}.
With these notation we can write \eqref{sec5-eq4} as
\ce
u^{\varepsilon,\delta}(t,x)=\mE^B\left(h(t\wedge\tau_D,X^{t,x}_{t\wedge\tau_D})\exp\{V_{t,\tau_D,x}^{\varepsilon,\delta}\}\right).
\de

We want to show that $u^{\varepsilon,\delta}(t,x)$
converges to $u(t,x)$ and by using a limiting argument to \eqref{sec5-eq3} to show that
$u(t,x)$ satisfies \eqref{e_mild_sol}.  We will do this in four steps.

\noindent {$\bf Step$ 1}. We will prove that for any $x\in D$ and any $t>0$, we have
\begin{equation}\label{sec5-eq5}
\lim_{\varepsilon,\delta\downarrow0}\mE^W|u^{\varepsilon,\delta}(t,x)-u(t,x)|^p=0
\end{equation}
for all $p\geq2$, where $u(t,x)$ is defined in (\ref{sec5-eq2}). Since $h$ is continuous bounded function, it suffices to show  \eqref{sec5-eq5} for the case $h\equiv1$. Since
\ce
\mE^W|u^{\varepsilon,\delta}(t,x)-u(t,x)|^p&=&\mE^W|\mE^B\exp\{V_{t,\tau_D,x}^{\varepsilon,\delta}\}-\mE^B\exp\{V_{t,\tau_D,x}\}|^p\\
&&\leq \mE|\exp\{V_{t,\tau_D,x}^{\varepsilon,\delta}\}-\exp\{V_{t,\tau_D,x}\}|^p\,,
\de
where $V_{t,\tau_D,x}$ is defined in (\ref{sec4-eq4}),  owing  to the fact that $\exp\{V_{t,\tau_D,x}^{\varepsilon,\delta}\}\xrightarrow{\mP} \exp\{V_{t,\tau_D,x}\}$ (by Theorem \ref{th4-1}), if we can prove $\{\exp\{ V_{t,\tau_D,x}^{\varepsilon,\delta}\}\}_{\varepsilon,\delta}$ is uniformly integrable in $L^p$, then  the proof of (\ref{sec5-eq5})   is completed.

To show
$\{\exp\{  V_{t,\tau_D,x}^{\varepsilon,\delta}\}\}_{\varepsilon,\delta}$ is uniformly integrable we only need to show that for  all  (in fact for some)    $\lambda >1$,
$\sup_{\varepsilon, \delta} \EE  \exp\{\lambda  V_{t,\tau_D,x}^{\varepsilon,\delta} \}
<\infty$.
In fact, by identity (\ref{sec4-eq3}),    inequalities (\ref{sec4-eq7}) and (\ref{sec4-eq9}), we see
\begin{eqnarray}
\mE\exp\{\lambda V_{t,\tau_D,x}^{\varepsilon,\delta}\}&=&\mE\exp\Big(\frac{\lambda^2}{2}\|A_{t\wedge\tau_D,x}^{\varepsilon,\delta}\|^2_{\cH}\Big) \nonumber\\
&\leq&\mE \Big(\frac{\lambda^2}{2}C_H\int_0^{t\wedge\tau_D}\int_0^{t\wedge\tau_D}|s-r|^{2H_0-2}\nonumber\\
&&\qquad\times\prod^d_{i=1}|X^{i,t,x}_s-X^{i,t,x}_r|^{2H_i-2}dsdr\Big) \nonumber \\
&<&\infty\, .\label{e.5.0}
\end{eqnarray}
This proves \eqref{sec5-eq5}.

\noindent {$ \bf Step$ 2}. We will now prove that $u(t,x)$ is a mild solution of (\ref{sec1-eq4}) in the sense of Definition (\ref{df5-2}). Identity \eqref{sec5-eq3} implies
\ce
u^{\varepsilon,\delta}(t,x)&=&\int_{D}f(y)p_D(t,x;0,y)dy+\int_0^t\int_Dp_D(t,x;s,y)u^{\varepsilon,\delta}(s,y)\dot{W}^{\varepsilon,\delta}(s,y)dyds\\
& &\qquad +\int_0^s\int_{\partial D}p_D(t,x;s,y)g(s,y)d{\bf S}(y)ds,
\de
where $p_D(t,x;s,y)$ is the fundamental solution of second order partial differential operator $\partial_t-L_t$
and $d{\bf S}(y)$ is the surface area on $\partial D$.
So as in \cite{Hu}, \cite{Hu1}, to prove that $u$ is a mild solution, it suffices to prove that
\begin{equation}\label{e.5.9}
\begin{split}
&\lim_{\varepsilon,\delta\downarrow0}\int_0^t\int_{D} p_D(t,x;s,y)u^{\varepsilon,\delta}(s,y) \dot{W}^{\varepsilon,\delta}(s,y)dyds\\
&\qquad=\int_0^t\int_{D} p_D(t,x;s,y)u(s,y) W(dy,ds)\quad \hbox{in probability} .
\end{split}
\end{equation}
We write
\ce
&&\int_0^t\int_{D}u^{\varepsilon,\delta}(s,y)p_D(t,x;s,y)\dot{W}^{\varepsilon,\delta}(s,y)dyds\\
&&\qquad\quad=\int_0^t\int_{D}(u^{\varepsilon,\delta}(s,y)-u(s,y))p_D(t,x;s,y)\dot{W}^{\varepsilon,\delta}(s,y)dyds\\
&&\qquad\quad \qquad\quad +\int_0^t\int_{D}u(s,y)p_D(t,x;s,y)\dot{W}^{\varepsilon,\delta}(s,y)dyds\\
&&\qquad\quad=:I_{\varepsilon,\delta}+J_{\varepsilon,\delta}\,.
\de
It is easy to show
\ce
\lim_{\varepsilon,\delta\downarrow0}J_{\varepsilon,\delta}=\int_0^t\int_{D}u(s,y)
p_D(t,x;s,y)W(dy,ds)\,.
\de
Thus, it suffices   to  prove
\begin{equation}\label{sec5-eq7}
\begin{split}
 \lim_{\varepsilon,\delta\downarrow0}I_{\varepsilon,\delta}:=
 &\lim_{\varepsilon,\delta\downarrow0}\int_0^t\int_{D}(u^{\varepsilon,\delta}(s,y)-u(s,y))p_D(t,x;s,y)\dot{W}^{\varepsilon,\delta}(s,y)dyds\\
 =&0
 \end{split}
\end{equation}
in $L^2(\Omega)$ 
in order for us to show  \eqref{e_mild_sol}.

The remaining part of this section is   to show \eqref{sec5-eq7}.
First,  we express the product $(u^{\varepsilon,\delta}(s,y)-u(s,y))\dot{W}^{\varepsilon,\delta}(s,y)$ in $I_{\varepsilon,\delta}$ as the sum of a divergence integral plus a trace term by formula (\ref{sec2-eq8}):
\[
\begin{split}
(u^{\varepsilon,\delta}(s,y)-u(s,y))&\dot{W}^{\varepsilon,\delta}(s,y) = (u^{\varepsilon,\delta}(s,y)-u(s,y))W(\varphi_\delta(s-\cdot)p_\varepsilon(y-\cdot))\\
 =&\delta((u^{\varepsilon,\delta}(s,y)-u(s,y))\varphi_\delta(s-\cdot)p_\varepsilon(y-\cdot))\\
&\qquad  +\langle\mD(u^{\varepsilon,\delta}(s,y)-u(s,y)),\varphi_\delta(s-\cdot)p_\varepsilon(y-\cdot)\rangle_{\cH}\\
 =&\int_0^t\int_{\mR^d}(u^{\varepsilon,\delta}(s,y)-u(s,y))\varphi_\delta(s-r)p_\varepsilon(y-z)\delta W_{r,z}
 \\
&\qquad  +\langle\mD(u^{\varepsilon,\delta}(s,y)-u(s,y)),\varphi_\delta(s-\cdot)p_\varepsilon(y-\cdot)\rangle_{\cH}\,,
\end{split}
\]
where $\delta W_{r,z}$ denotes the It\^o-Skorohod integral.
Thus, we have
\begin{eqnarray}
I_{\varepsilon,\delta} &=& \int_0^t\int_{D}\int_0^t\int_{\mR^d}p_D(t,x;s,y)(u^{\varepsilon,\delta}(s,y)-u(s,y))\varphi_\delta(s-r)p_\varepsilon(y-z)\delta W_{r,z}dyds \nonumber\\
&&\qquad  +\int_0^t\int_{D}p_D(t,x;s,y)\langle\mD(u^{\varepsilon,\delta}(s,y)-u(s,y)),\varphi_\delta(s-\cdot)p_\varepsilon(y-\cdot)\rangle_{\cH}dyds \nonumber\\
 &=& \int_0^t\int_{\mR^d}\int_0^t\int_{D}p_D(t,x;s,y)(u^{\varepsilon,\delta}(s,y)-u(s,y))\varphi_\delta(s-r)p_\varepsilon(y-z)dyds\delta W_{r,z}\nonumber\\
&&\qquad  +\int_0^t\int_{D}p_D(t,x;s,y)\langle\mD(u^{\varepsilon,\delta}(s,y)-u(s,y)),\varphi_\delta(s-\cdot)p_\varepsilon(y-\cdot)\rangle_{\cH}dyds \nonumber\\
 &:= & \int_0^t\int_{\mR^d}\phi^{\varepsilon,\delta}_{r,z}\delta W_{r,z}+I_{2,\varepsilon,\delta}:=I_{1,\varepsilon,\delta}+I_{2,\varepsilon,\delta} \,, \label{sec5-eq8}
\end{eqnarray}
where
\ce
\phi^{\varepsilon,\delta}_{r,z}=\int_0^t\int_{D}p_D(t,x;s,y)(u^{\varepsilon,\delta}(s,y)-u(s,y))\varphi_\delta(s-r)p_\varepsilon(y-z)dyds
\de
and $\delta(\phi^{\varepsilon,\delta})=\int_0^t\int_{\mR^d}\phi^{\varepsilon,\delta}_{r,z}\delta W_{r,z}$ denotes the divergence (It\^o-Skorohod integral) of $\phi^{\varepsilon,\delta}$ (see (\ref{sec2-eq7})).

\noindent {$\bf Step$ 3}. In this step we claim that
\ce
\lim_{\varepsilon,\delta\downarrow0}I_{1,\varepsilon,\delta}=0 ~~in~~L^2(\Omega).
\de
We use  the following  inequality \cite[1.47]{Nualart} to bound  the Skorohod integral
\begin{equation}\label{e.5.10}
\mE[(I_{1,\varepsilon,\delta})^2]\leq\mE(||\phi^{\varepsilon,\delta}||^2_{\cH})+\mE(||\mD\phi^{\varepsilon,\delta}||^2_{\cH\otimes\cH}).
\end{equation}
First, we  handle the first term in (\ref{e.5.10}). We have
\begin{eqnarray}
&&\mE(||\phi^{\varepsilon,\delta}||^2_{\cH}) = \mE(\langle\phi^{\varepsilon,\delta}_{r,z},\phi^{\varepsilon,\delta}_{r,z}\rangle_{\cH}) \nonumber \\
&&\qquad = \int_0^t\int_{D}\int_0^t\int_{D}\mE((u^{\varepsilon,\delta}(s,y)-u(s,y))(u^{\varepsilon,\delta}(\rho,\xi)-u(\rho,\xi))) \nonumber \\
&&\qquad\qquad   \times p_D(t,x;s,y)p_D(t,x;\rho,\xi)\nonumber \\
&&\qquad\qquad \langle\varphi_\delta(s-\cdot)p_\varepsilon(y-\cdot), \varphi_\delta(\rho-\cdot)p_\varepsilon(\xi-\cdot)\rangle_{\cH}dydsd\xi d\rho. \label{sec5-eq9}
\end{eqnarray}
Using \cite[Lemma A.2, Lemma A.3]{Hu}, we can write
\begin{eqnarray}
& \langle\varphi_\delta(s-\cdot)p_\varepsilon(y-\cdot), \varphi_\delta(\rho-\cdot)p_\varepsilon(\xi-\cdot)\rangle_{\cH} \nonumber \\
& \qquad \leq C_H|s-\rho|^{2H_0-2}\prod^d_{i=1}|\xi_i-y_i|^{2H_i-2}\label{sec5-eq10}
\end{eqnarray}
for some constant $C_H>0$.  Due to  (\ref{sec5-eq5})
\ce
&&\mE\left((u^{\varepsilon,\delta}(s,y)-u(s,y))(u^{\varepsilon,\delta}(l,v)-u(l,v))\right)\\
&\qquad& \leq\mE^B\left((\mE^W|u^{\varepsilon,\delta}(s,y)-u(s,y)|^2)^{1/2}(\mE^W|u^{\varepsilon,\delta}(l,v)-u(l,v)|^2)^{1/2}\right)\\
&\qquad& \rightarrow 0,~~~as~~~\varepsilon,\delta\rightarrow0.
\de
As a consequence, the integrand on the right-hand side of (\ref{sec5-eq9}) converges to zero as $\varepsilon$ and $\delta$ tend to zero for any $s, \rho, y, \xi$.

On the other hand,  from (\ref{e.5.0}) it follows
\ce
&&\sup_{\varepsilon,\delta}\sup_{y\in D}\sup_{0\leq s\leq t}\mE(u^{\varepsilon,\delta}(s,y))^2\\
&&\qquad =\sup_{\varepsilon,\delta}\sup_{y\in D}\sup_{0\leq s\leq t}\mE\exp\Big\{2V^{\varepsilon,\delta}_{s\wedge\tau_D,y}\Big\}<\infty
\de
and
\ce
&&\int_0^t\int_{D}\int_0^t\int_{D}p_D(t,x;s,y)p_D(t,x;\rho,\xi)|s-\rho|^{2H_0-2}\prod^d_{i=1}|\xi_i-y_i|^{2H_i-2}dydsd\xi d\rho\\
& &\qquad =\int_0^t\int_0^t|s-\rho|^{2H_0-2}\int_{D}\int_{D}p_D(t,x;s,y)p_D(t,x;\rho,\xi)\prod^d_{i=1}|\xi_i-y_i|^{2H_i-2}dyd\xi d\rho ds\\
& &\qquad \leq  \int_0^t\int_0^t|s-\rho|^{2H_0-2} \mE[\prod^d_{i=1}|B_{t-s}^{i,x}-B_{t-\rho}^{i,x}|^{2H_i-2}]d\rho ds\\
& &\qquad\leq\int_0^t\int_0^t|s-\rho|^{2H_0-2} \mE[\prod^d_{i=1}|B_{\rho-s}^{i,x}|^{2H_i-2}]d\rho ds\\
& &\qquad  \leq C_{H_i}\int_0^t\int_0^t|s-\rho|^{2H_0-2} (s-\rho)^{\sum^d_{i=1}H_i-d}d\rho ds\\
& &\qquad =2C_{H_i}\int_0^t\int_0^s(s-\rho)^{2H_0+\sum^d_{i=1}H_i-d-2}d\rho ds\\
& &\qquad =2C_{H_i}t^{2H_0+\sum^d_{i=1}H_i-d}
\de
when $2H_0+\sum^d_{i=1}H_i-d-1>0$. Therefore, applying  the dominated convergence  theorem
to \eqref{sec5-eq9}   yields  that $\lim_{\varepsilon,\delta\rightarrow0}\mE(||\phi^{\varepsilon,\delta}||^2_{\cH})\rightarrow0$.

Now we deal with  the second term on the right hand side of  (\ref{e.5.10}).   By \cite[Proposition 1.2.3]{Nualart},  the Malliavin derivative of $V^{\varepsilon,\delta}_{t,\tau_D,x}$ is
\ce
\mD V^{\varepsilon,\delta}_{t,\tau_D,x}&=&\mD\Big(\int_0^{\infty}\int_{\mR^d}A^{\varepsilon,\delta}_{t\wedge\tau_D,x}(r,y)W(dy,dr)\Big)=A^{\varepsilon,\delta}_{t\wedge\tau_D,x}(\cdot,\cdot)\,.
\de
With the help of chain rule  of the Malliavin derivative $\mD$ with respect to the path of $W(t,x)$  we  see
\begin{eqnarray}\label{e.5.11}
\mD(u^{\varepsilon,\delta}(t,x))=\mE^B(\exp\{V^{\varepsilon,\delta}_{t,\tau_D,x}\}A^{\varepsilon,\delta}_{t\wedge\tau_D,x})
\end{eqnarray}
and
\ce
\mD \phi^{\varepsilon,\delta}_{r,z}=\int_0^t\int_{D}p_D(t,x;s,y)\mD(u^{\varepsilon,\delta}(s,y)-u(s,y))\varphi_\delta(s-r)p_\varepsilon(y-z)dyds\,. \label{e.4.15}
\de
Therefore,
\begin{eqnarray}\label{sec5-eq11}
&&\mE\langle \mD u^{\varepsilon,\delta}(t,x),\mD u^{\varepsilon,\delta}(t,x)\rangle_{\cH} \nonumber\\
& &\qquad =\mE^W\mE^B\Big(\exp\{V^{\varepsilon,\delta}_{t,\tau_D,x}(X^{1,t,x})+V^{\varepsilon,\delta}_{t,\tau_D,x}(X^{2,t,x})\}\nonumber\\
&&\qquad \qquad
\langle A^{\varepsilon,\delta}_{t\wedge\tau_D,x}(X^{1,t,x}),A^{\varepsilon,\delta}_{t\wedge\tau_D,x}(X^{2,t,x})\rangle_{\cH}\Big) \,,
\end{eqnarray}
where $X^1$ and $X^2$ are two independent copies of the
 $d$-dimensional solutions of equation (\ref{sec2-eq01}). Then, from  \eqref{sec4-eq6}  it
 follows
\begin{eqnarray}\label{sec5-eq12}
&&\lim_{\varepsilon,\delta\rightarrow0}\langle A^{\varepsilon,\delta}_{t\wedge\tau_D,x}(X^{1,t,x}),A^{\varepsilon,\delta}_{t\wedge\tau_D,x}(X^{2,t,x})\rangle_{\cH}  \nonumber\\
& &\qquad= \alpha_H\lim_{\varepsilon,\delta\rightarrow0}\int_{[0,t]^2}\int_{[0,t\wedge\tau_D]^2}\int_{\mR^{2d}}p_\varepsilon(X^{1,t,x}_u-y)p_{\varepsilon}(X^{2,t,x}_v-z)\nonumber\\
&&\qquad\qquad \times
\varphi_\delta(t-r-u)\varphi_{\delta}(t-s-v)|r-s|^{2H_0-2}\times\prod^d_{i=1}|y_i-z_i|^{2H_i-2}dydzdudvdrds
\nonumber\\
& &\qquad=\alpha_H\int_{[0,t\wedge\tau_D]^2}|s-r|^{2H_0-2}\times\prod^d_{i=1}|X^{1,i,t,x}_s-X^{2,i,t,x}_r|^{2H_i-2}dsdr.
\end{eqnarray}

Thus by \eqref{e.4.15}, we have
\begin{eqnarray}
&&\mE(||\mD\phi^{\varepsilon,\delta}||^2_{\cH\otimes\cH}) \nonumber\\
&&=\int_0^t\int_{D}\int_0^t\int_{D}\mE\langle\mD(u^{\varepsilon,\delta}(s,y)-u(s,y)),\mD(u^{\varepsilon,\delta}(l,v)-u(l,v)) \rangle_{\cH}  \nonumber\\
&&\qquad \times p_D(t,x;s,y)p_D(t,x;l,v) \langle \varphi_\delta(s-)p_\varepsilon(y-),
\varphi_\delta(l-)p_\varepsilon(v-)\rangle_{\cH}dydsdvdl\,. \label{sec5-eq14}
\end{eqnarray}
The integrand   $\mE\langle\mD(u^{\varepsilon,\delta}(t,x)),\mD(u^{\varepsilon,\delta}(t,x)) \rangle_{\cH}$ in the above integral can be estimated in exactly the same way as   (\ref{sec5-eq12})
\ce
&&\lim_{\varepsilon,\delta\rightarrow0}\mE\langle\mD(u^{\varepsilon,\delta}(t,x)),\mD(u^{\varepsilon,\delta}(t,x)) \rangle_{\cH}\\
&&\qquad =\lim_{\varepsilon,\delta\rightarrow0}\mE^B\Big[\mE^W\Big(\exp\{V^{\varepsilon,\delta}_{t,\tau_D,x}(X^{1,t,x})+V^{\varepsilon,\delta}_{t,\tau_D,x}(X^{2,t,x})\}\Big)\\
&&\qquad\qquad \times \alpha_H\int_{[0,t\wedge\tau_D]^2}|s-r|^{2H_0-2}\times\prod^d_{i=1}|X^{1,i,t,x}_s-X^{2,i,t,x}_r|^{2H_i-2}dsdr\Big]\\
&&\qquad =\mE^B\Big[\exp \left\{\frac{\alpha_H}{2}\sum^2_{j,k=1}\int_0^{t\wedge\tau_D}\int_0^{t\wedge\tau_D}|s-r|^{2H_0-2}\prod^d_{i=1}|X_s^{j,i,t,x}-X_r^{k,i,t,x}|^{2H_i-2}dsdr\right\} \\
&&\qquad \qquad\times \alpha_H\int_0^{t\wedge\tau_D}\int_0^{t\wedge\tau_D}|s-r|^{2H_0-2}\prod^d_{i=1}|X_s^{1,i,t,x}-X_r^{2,i,t,x}|^{2H_i-2}dsdr\Big].
\de
This implies that $u^{\varepsilon,\delta}(t,x)$ converges in $\mD^{1,2}$ to $u(t,x)$ as $\varepsilon,\delta\rightarrow0$. In order to taking the limit (\ref{sec5-eq14}) inside the integral, we need uniformly integrability of $\mD(u^{\varepsilon,\delta}(s,y))$ with respect to $\varepsilon,\delta$.
In fact, using (\ref{sec5-eq11}) and  (\ref{e.5.0}), we have
\ce
&&\sup_{\varepsilon,\delta}\sup_{y\in D}\sup_{0\leq s\leq t}\mE\|\mD(u^{\varepsilon,\delta}(s,y))\|_{\cH}^2\\
&&\leq \sup_{\varepsilon,\delta}\sup_{y\in D}\sup_{0\leq s\leq t}\mE^W\mE^B\Big(\exp\{V^{\varepsilon,\delta}_{s,\tau_D,y}(X^1)+V^{\varepsilon,\delta}_{s,\tau_D,y}(X^2)\} \nonumber\\
&&\qquad \times \alpha_H\int_{[0,s\wedge\tau_D]^2}|u-v|^{2H_0-2}\times\prod^d_{i=1}|X^{i,1,t,y}_u-X^{i,2,t,y}_v|^{2H_i-2}dudv\Big) \nonumber\\
&&\leq  \sup_{y\in D}\sup_{0\leq s\leq t}\mE^B\Big[\exp\bigg \{\frac{\alpha_H}{2}\sum^2_{j,k=1}\int_0^{s\wedge\tau_D}\int_0^{s\wedge\tau_D}|u-v|^{2H_0-2}\\
&&\qquad\qquad \prod^d_{i=1}|X_u^{i,j,t,y}-X_v^{i,k,t,y}|^{2H_i-2}dudv\bigg\} \\
&&\qquad \times \alpha_H\int_0^{s\wedge\tau_D}\int_0^{s\wedge\tau_D}|u-v|^{2H_0-2}\prod^d_{i=1}
 |X_u^{i,1,t,y}-X_v^{i,2,t,y}|^{2H_i-2}  dudv\Big] \\
 &&\leq  C_1 \sup_{y\in D}\sup_{0\leq s\leq t}\mE^B\Big[\exp\left\{c_2 \sum^2_{j,k=1}\int_0^{s }\int_0^{s }|u-v|^{2H_0-2}\prod^d_{i=1}|B_u^{i,1  }-B_v^{i,2 }|^{2H_i-2}dudv\right\} \\
 &&\qquad \times  \int_0^{s\wedge\tau_D}\int_0^{s\wedge\tau_D}|u-v|^{2H_0-2}\prod^d_{i=1}
 |B_u^{i,1 }-B_v^{i,2 }|^{2H_i-2}  dudv\Big]
\de
which is finite by \cite{Hu},  where the last inequality follows from Lemma \ref{l.3.1}.
This proves that $\lim_{\varepsilon,\delta\downarrow0}I_{1,\varepsilon,\delta}=0 ~~in~~L^2(\Omega)$.

\noindent {$\bf Step$ 4}. In this step we show
\ce
\lim_{\varepsilon,\delta\downarrow0}I_{2,\varepsilon,\delta}=0 ~~in~~L^2(\Omega)\,,
\de
where $I_{2,\varepsilon,\delta}$ is defined in \eqref{sec5-eq8}.
We first write the Malliavin calculus of $u(t,x)$ defined in (\ref{sec5-eq2}).  By  (\ref{sec4-eq4}),   we have
\ce
\mD u(t,x)&=&\mE^B(\exp\{V_{t,\tau_D,x}\}\cdot\mD V_{t,\tau_D,x})\\
&=&\mE^B\Big(\exp\{V_{t,\tau_D,x}\}(\delta (X^{t,x}_r-\cdot)\cdot\mI_{\{r\leq t\wedge\tau_D\}}
\Big).
\de
Thus
\ce
I_{2,\varepsilon,\delta}&=&\int_0^t\int_{D}p_D(t,x;s,y)\langle\mD(u^{\varepsilon,\delta}(s,y)-u(s,y)),\varphi_\delta(s-\cdot)p_\varepsilon(y-\cdot)\rangle_{\cH}dyds\\
&=&\int_0^t\int_{D}p_D(t,x;s,y)\mE^B\Big(\exp\{V^{\varepsilon,\delta}_{s,\tau_D,y}\}\langle A^{\varepsilon,\delta}_{s\wedge\tau_D,y},\varphi_\delta(s-\cdot)p_\varepsilon(y-\cdot)\rangle_{\cH}\Big)dyds \\
&&-\int_0^t\int_{D}p_D(t,x;s,y)\mE^B\Big(\exp\{V_{s,\tau_D,y}\}\\
&&\qquad  \times \langle \delta (X^{t,y}_{\cdot}-\cdot)\cdot\mI_{\{\cdot\leq s\wedge\tau_D\}},\varphi_\delta(s-\cdot)p_\varepsilon(y-\cdot)\rangle_{\cH}\Big)dyds\\
&=: & I^1_{2,\varepsilon,\delta}-I^2_{2,\varepsilon,\delta},
\de
where by (\ref{sec2-eq3}) and (\ref{sec4-eq2}) we get
\ce
&&\langle A^{\varepsilon,\delta}_{s\wedge\tau_D,y},\varphi_\delta(s-\cdot)p_\varepsilon(y-\cdot)\rangle_{\cH}\\
&&\qquad =\alpha_H\int_0^{s\wedge\tau_D}\int_{\mR^2_+\times\mR^{2d}}\varphi_\delta(s-u-r)p_\varepsilon(X^{t,y}_r-w)\\
&&\qquad \qquad  \times \varphi_\delta(s-v)p_\varepsilon(y-z) |u-v|^{2H_0-2}\prod^d_{i=1}|w_i-z_i|^{2H_i-2}dwdudzdvdr.
\de
By \cite[Lemma A.2 and Lemma A.3]{Hu}, we have
\ce
&&\langle A^{\varepsilon,\delta}_{s\wedge\tau_D,y},\varphi_\delta(s-\cdot)p_\varepsilon(y-\cdot)\rangle_{\cH}\\
&&\qquad \leq C_H\int_0^{s\wedge\tau_D} r^{2H_0-2}\prod^d_{i=1}|X^{i,t,y}_r|^{2H_i-2}dr.
\de
Therefore, we obtain the inner product estimate inside  $I^1_{2,\varepsilon,\delta}$,
\begin{eqnarray}\label{sec5-eq15}
\langle A^{\varepsilon,\delta}_{s\wedge\tau_D,y},\varphi_\delta(s-\cdot)p_\varepsilon(x-\cdot)\rangle_{\cH}
\leq C_H\int_0^{s\wedge\tau_D} r^{2H_0-2}\prod^d_{i=1}|X^{i,t,y}_r|^{2H_i-2}dr.
\end{eqnarray}
Similarly, we can get the inner product estimate inside  $I^2_{2,\varepsilon,\delta}$
\begin{eqnarray}\label{sec5-eq16}
&&\langle \delta (X^{t,y}_{\cdot}-\cdot)\cdot\mI_{\{\cdot\leq s\wedge\tau_D\}},\varphi_\delta(s-\cdot)p_\varepsilon(y-\cdot)\rangle_{\cH} \nonumber\\
&&=\alpha_H\int_{\mR^2_+\times\mR^{2d}}\delta (X^{t,y}_{u}-w)\varphi_\delta(s-v)p_\varepsilon(y-z) \nonumber\\
&&\qquad \times |u-v|^{2H_0-2}\prod^d_{i=1}|w_i-z_i|^{2H_i-2}\cdot\mI_{\{u\leq s\wedge\tau_D\}}dwdzdudv   \nonumber\\
&&=\alpha_H\int_{\mR^2_+\times\mR^d}\varphi_\delta(s-v)p_\varepsilon(y-z) |u-v|^{2H_0-2}\prod^d_{i=1}|X^{i,t,y}_{u}-z_i|^{2H_i-2}\cdot\mI_{\{u\leq s\wedge\tau_D\}}dzdudv   \nonumber\\
&&\leq C_H\int_0^{s\wedge\tau_D} r^{2H_0-2}\prod^d_{i=1}|X^{i,t,y}_r|^{2H_i-2}dr.
\end{eqnarray}
Then, from (\ref{sec5-eq15}), (\ref{sec5-eq16}), and from
the fact that the random variable $\int_0^s r^{2H_0-2}\prod^d_{i=1}|X^{i,t,y}_r|^{2H_i-2}dr$ is square integrable because of Lemma \ref{l.3.1} it follows
\ce
&&\int_0^t\int_Dp(t,x;s,y)\left(\mE(\int_0^{s\wedge\tau_D} r^{2H_0-2}\prod^d_{i=1}|X^{i,t,y}_r|^{2H_i-2}dr)^2\right)^{1/2}dyds\\
& &\qquad\le   \int_0^t\int_Dp(t,x;s,y) s^{2H_0+\sum^d_{i=1}H_i-d-1}dyds\\
& &\qquad \le  \int_0^t s^{2H_0+\sum^d_{i=1}H_i-d-1}ds<\infty.
\de
Therefore, we can apply the dominated convergence theorem to obtain  that
\ce
&&\lim_{\varepsilon,\delta\rightarrow0}I^1_{2,\varepsilon,\delta}
=\lim_{\varepsilon,\delta\rightarrow0}I^2_{2,\varepsilon,\delta}\\
&&\qquad =\alpha_H\int_0^t\int_{D}p_D(t,x;s,y)
\mE^B\Big(\exp\{V_{s,\tau_D,y}\}\int_0^{s,\tau_D,y} r^{2H_0-2}\prod^d_{i=1}|X^{i,t,y}_r|^{2H_i-2}dr\Big)dyds
\de
in $L^2(\Omega)$.   This gives  $\lim_{\varepsilon,\delta\rightarrow0}
I_{2,\varepsilon,\delta}=0$.    The proof of the theorem is  completed.
\end{proof}

If the Skorohod integral is used in \eqref{sec1-eq4}  then we have  a similar theorem, whose proof is similar and is omitted.

\begin{theorem}\label{th5-2}  Suppose that ({\bf H2.1}) hold true for the coefficients $b$ and $\si$ and $2H_0+\sum^d_{i=1}H_i-d-1>0$ and that $f, g$ are continuous bounded measurable functions. Let $X_\cdot^{t,x}$ satisfy \eqref{sec1-eq2}.
Then, the random field
\begin{equation}
	\begin{cases}
		u(t,x)     = \mE^B\Big(f(X^{t,x}_t)\mI_{t\le \tau_D}\exp\Big\{A(t, t ,x) \Big\}\\
		\qquad\qquad \qquad\qquad +g(\tau_D,X^{t,x}_{\tau_D})\mI_{t>\tau_D}\exp\Big\{A(\tau_D, t ,x)\Big\}\Big)\,;
		\\ \\
		A(r, t ,x)  =\int_{t-r}^t   W(X_{t-s}^{t,x} ,ds)\\
		\qquad\qquad \qquad\qquad -\frac12 \int_{t-r}^t\int_{t-r}^t
		|s_2-s_1|^{2H_0-2}   \phi_H( X_{t- s_2}^{t,x}-X_{t- s_1}^{t,x}) ds_1ds_2\\
			\qquad\qquad =\int_{t-r}^t \int_{\mR^d}\delta (X_{t-s}^{t,x}-y)W(dy,ds)\\
		\qquad\qquad \qquad\qquad -\frac12 \int_0^r \int_0^r
		|s_2-s_1|^{2H_0-2}   \phi_H( X_{  s_2}^{t,x}-X_{  s_1}^{t,x}) ds_1ds_2\,.
	\end{cases}
	\label{sec1-eq04b}
\end{equation}
is a mild solution of (\ref{sec1-eq4}),
where
\begin{equation}
h(t,x)=\begin{cases}
f(x)  & \qquad \hbox{when}\  t\ge 0\,, x\in D\\
g(t, x) & \qquad \hbox{when}\  t\ge 0\,, x\in \partial D\,.\\
\end{cases}
\end{equation}
\end{theorem}

\section{
H\"older continuity of solution}

In this section we shall obtain the
H\"older continuity of the solution. Notice that for different $t$, $X_s^{t,x}$
satisfies different stochastic differential equations
with different coefficients. This requires us to know  the dependence of the solution
$X_s^{t,x}$ on $t,x$. For this reason we need to work on the same probability space and hence we need to consider the strong solutions.
To this end we shall  assume  more  regularity on the  coefficients.

\smallskip
\begin{assumption}\label{a.6.1}
\begin{enumerate}
\item[(i)]\ $b, \sigma$ are    Lipschitz in $x$ uniformly in $t$, namely,      there is a positive constant $K_1$ such that
\ce
&&|b(t,x_1)-b(t,x_2)|\leq K_1 |x_1-x_2|,
\quad   |\sigma(t,x_1)-\sigma(t,x_2)|\leq K_1 |x_1-x_2|
\de
if $ x_1 , x_2\in D$, $0\leq t \leq T$.
\item[(ii)]\ $b, \sigma$ are   H\"older continuous in $t$ with H\"older exponent $\gamma>0$ uniformly in $x\in D$.
Namely,    there is a positive constant $K_2 $ such that
\ce
&&|b(t_1,x )-b(t_2,x )|\leq K_2 |t_1-t_2|^\gamma,
\quad   |\sigma(t_1,x )-\sigma(t_2,x )|  \leq K_2 |t_1-t_2|^\gamma
\de
if $ x \in D$, $0\leq t_1,  t_2\leq T$.
\end{enumerate}
\end{assumption}
Under the above conditions we have
\begin{theorem}The strong solution
$X_\cdot^{t,x}$ exists uniquely. Moreover, we have
for any $p\ge 2$,
\begin{equation}
\EE\sup_{0\le r\le T} \left[|X_r^{t,x} -X_r^{s,x} |^p\right]
\le C_p |t-s|^{\gamma p} \label{e.6.1a}
\end{equation}
and for any $0\le u\le v\le T$,
\begin{equation}
\EE  \left[|X_u^{t,x} -X_u^{s,x}
-(X_v^{t,x} -X_v^{s,x} )|^p\right]
\le C_p |u-v|^{p/2} |t-s|^{\gamma p}\,.
 \label{e.6.2a}
\end{equation}
\end{theorem}
\begin{proof}The proof is elementary and we sketch it.
Equation \eqref{sec2-eq01} means
\[
X_r^{t,x}=x+\int_0^r b(t-\xi, X_\xi ^{t,x})d\xi
+\int_0^r \si(t-\xi, X_\xi ^{t,x})dB_\xi
\]
and
\[
X_r^{s,x}=x+\int_0^r b(s-\xi, X_\xi ^{s,x})d\xi
+\int_0^r \si(s-\xi, X_\xi ^{s,x})dB_\xi \,.
\]
Subtract these two equations to obtain
\begin{eqnarray*}
&&X_r^{t,x}
-X_r^{s,x}\\
&&\quad = \int_0^r \left[
b(t-\xi, X_\xi ^{t,x})-b(s-\xi, X_\xi ^{s,x})\right]  d\xi
+\int_0^r \left[ \si(t-\xi, X_\xi ^{t,x})
-\si(s-\xi, X_\xi ^{s,x})\right] dB_\xi.
 \end{eqnarray*}
Thus,  for all $r\in [0, T]$,
\begin{eqnarray*}
&&
\EE\sup_{0\le \xi\le r} |X_\xi^{t,x}
-X_\xi^{s,x}|^p \\
&&\quad \le
 C_{p, T} |t-s|^{ \ga p} +\EE \int_0^r \left|
 X_\xi ^{t,x}- X_\xi ^{s,x})\right|^p   d\xi\\
 &&\qquad\qquad
+C_p \EE \left[\int_0^r \left|\si(t-\xi, X_\xi ^{t,x})
-\si(s-\xi, X_\xi ^{s,x})\right|^2 d \xi  \right]^{p/2} \\
&&\quad \le
 C_{p, T} |t-s|^{ \ga p} +\EE \int_0^r \left|
 X_\xi ^{t,x}- X_\xi ^{s,x})\right|^p   d\xi \,.
 \end{eqnarray*}
An application of Gronwall lemma yields \eqref{e.6.1a}.
To prove \eqref{e.6.2a}, we write
\begin{eqnarray*}
&&X_v^{t,x}
-X_u^{t,x}-(X_v^{s,x}
-X_u^{s,x}) \\
&&\quad = \int_u^v \left[
b(t-\xi, X_\xi ^{t,x})-b(s-\xi, X_\xi ^{s,x})\right]  d\xi
+\int_u^v \left[ \si(t-\xi, X_\xi ^{t,x})
-\si(s-\xi, X_\xi ^{s,x})\right] dB_\xi.
 \end{eqnarray*}
Thus,
\begin{eqnarray*}
&&\EE\left|X_v^{t,x}
-X_u^{t,x}-(X_v^{s,x}
-X_u^{s,x})\right|^p  \\
&&\quad \le C_{p, T} (v-u)^{p-1}  \int_u^v \EE \left|
b(t-\xi, X_\xi ^{t,x})-b(s-\xi, X_\xi ^{s,x})\right|^p   d\xi \\
&&\qquad\quad +C_p
\EE \left(\int_u^v \left|[ \si(t-\xi, X_\xi ^{t,x})
-\si(s-\xi, X_\xi ^{s,x})\right|^2d\xi\right)^ {p/2}  \\
&&\quad \le C_{p, T} (v-u)^{p-1}  \int_u^v \EE \left|
b(t-\xi, X_\xi ^{t,x})-b(s-\xi, X_\xi ^{s,x})\right|^p   d\xi \\
&&\qquad\quad +C_p (v-u)^{p/2-1}
 \int_u^v \EE \left|[ \si(t-\xi, X_\xi ^{t,x})
-\si(s-\xi, X_\xi ^{s,x})\right|^pd\xi   \,.
 \end{eqnarray*}
Using Assumption  {(\bf 6.1)}  and the obtained
result \eqref{e.6.1a} the inequality \eqref{e.6.2a} follows easily.
\end{proof}

Denote
\[
\varrho=2H_0+\sum^d_{i=1}H_i-d-1
\]
and
\[
\varrho'=\min\left(\varrho,  \frac{6H_i-5}{2H_i-1}\right).
\]
\begin{theorem}\label{th6-1}
Suppose that Assumption \ref{a.6.1}  holds, $\min(H_1, \cdots, H_d)>5/6$,   and $2H_0+\sum^d_{i=1}H_i-d-1>0$. Suppose the initial condition
$f:D\to\RR$ and $g:[0, T]\times \partial D$ are Lipschitz continuous.     Then, the Feynman-Kac solution $u(t,x)$
(given by \eqref{sec5-eq2})   of (\ref{sec1-eq4}) has a H\"older continuous modification.
More precisely,  for any $\beta\in (0, \frac{\varrho}{2})$,  $\beta'\in (0, \frac{\varrho'}{2})$, and   any compact rectangle $I\subset \mR_+\times \mR^d$,    there exists a positive random variable $K_I$ such that almost surely, for any $(t,x), (s,y)\in I$, we have
\ce
|u(t,x)-u(s,y)|\leq K_I(|t-s|^{\beta'}+|x-y|^{2\beta}).
\de
\end{theorem}
\begin{proof} For the space H\"older exponent,  we shall adopt the   proof of Theorem 5.1 in \cite{Hu} with the modification so that it is still valid with the Brownian
	motion replaced by the diffusion process $X$.  In fact, we need to prove that the  estimates (5.4) and (5.5) in \cite{Hu} still hold  true when  the  Brownian motion
in that paper is replaced by $X^{t,x}_{\cdot}$,  the solution of equation (\ref{sec2-eq01}) starting at $x$ (at time $0$).  We will pay specific
attention to this difference.  For the time H\"older exponent,
due to the time dependence of the coefficients,  things are much more complicated since we shall encounter  some more singularity in the integrations. We need to use some trick to handle it.

To simplify notation and  without loss of generality we can assume $f=1$ and $h=1$ and in this case we have
\begin{eqnarray}
u(t,x)&=&\mE^B\left( \exp\Big\{ V_{t, \tau_D,x}  \Big\} \right)  \,,
 \label{sec6-eq1}\\
 V_{t, \tau_D,x}&=&  \int_0^{t\wedge \tau_D} \int_{\mR^d}\delta (X_s^{t,x}-y)W(dy,ds)\,. \nonumber
\end{eqnarray}
First,  we have
\ce
&&\mE|V_{t,\tau_D,x}-V_{s,\tau_D,y}|^2\\
&&\qquad =\alpha_H\mE^B\Big(\int_0^{t\wedge\tau_D}\int_0^{t\wedge\tau_D}|u-v|^{2H_0-2}\prod^d_{i=1}|X^{i,t,x}_{u}-X^{i,t,x}_{v}|^{2H_i-2}dudv\\
&&\qquad\qquad\qquad +\int_0^{s\wedge\tau_D}\int_0^{s\wedge\tau_D}|u-v|^{2H_0-2}\prod^d_{i=1}|X^{i,s,y}_{u}-X^{i,s,y}_{v}|^{2H_i-2}dudv\\
&&\qquad\qquad\qquad-2\int_0^{t\wedge\tau_D}\int_0^{s\wedge\tau_D}|u-v|^{2H_0-2}\prod^d_{i=1}|X^{i,t,x}_{u}-X^{i,s,y}_{v}|^{2H_i-2}dudv\Big)\\
&&\qquad:=\alpha_H A(s,t,x,y,\tau_D)\,.
\de
We divide our study of the above quantity into two cases.

\noindent
{\bf Case (i)}:  When $s=t$ and $x\neq y$.
In this case, we have
\begin{eqnarray*}
A(t,t,x,y,\tau_D)&=&\mE^B\Big(\int_0^{t\wedge\tau_D}\int_0^{t\wedge\tau_D}|u-v|^{2H_0-2}\Big(\prod^d_{i=1}|X^{i,t,x}_u-X^{i,t,x}_v|^{2H_i-2}
\nonumber \\
&&\qquad \qquad  -\prod^d_{i=1}|X^{i,t,x}_u-X^{i,t,y}_v|^{2H_i-2}\Big)dudv\Big)\nonumber\\
&&\qquad +\mE^B\Big(\int_0^{t\wedge\tau_D}\int_0^{t\wedge\tau_D}|u-v|^{2H_0-2}\Big(\prod^d_{i=1}|X^{i,t,y}_u-X^{i,t,y}_v|^{2H_i-2}\\
&&\qquad \qquad  -\prod^d_{i=1}|X^{i,t,x}_u-X^{i,t,y}_v|^{2H_i-2}\Big)dudv\Big)\nonumber\\
&=& A_1(t,t,x,y,\tau_D)+A_2(t,t,x,y,\tau_D)\,.
\end{eqnarray*}
where    $A_1(t,t,x,y,\tau_D)$ and $A_2(t,t,x,y,\tau_D)$ can be treated in the same way.  We can write
\begin{eqnarray}
A_1(t,t,x,y,\tau_D)
&=& \sum_{j=1}^d \mE^B\Big(\int_0^{t\wedge\tau_D}\int_0^{t\wedge\tau_D}|u-v|^{2H_0-2}\prod^{j-1} _{i=1}|X^{i,t,x}_u-X^{i,t,x}_v|^{2H_i-2}\nonumber\\
&&\qquad \qquad  \times\prod _{i= j+1}^d |X^{i,t,x}_u-X^{i,t,y}_v|^{2H_i-2}    
\nonumber\\
&&\qquad\qquad \qquad
\Big(|X^{j,t,x}_u-X^{j,t,x}_v|^{2H_j-2}-|X^{j,t,x}_{u}-X^{j,t,y}_{v}|^{2H_j-2}\Big)
\Big)dudv\Big)\nonumber\\
&\le & \sum_{j=1}^d \mE^B\Big(\int_0^{t }\int_0^{t }|u-v|^{2H_0-2}\prod^{j-1} _{i=1}|X^{i,t,x}_u-X^{i,t,x}_v|^{2H_i-2}\prod _{i= j+1}^d |X^{i,t,x}_u-X^{i,t,y}_v|^{2H_i-2} \nonumber\\
&&\qquad \qquad  \times
 \Big||X^{j,t,x}_u-X^{j,t,x}_v|^{2H_j-2}-|X^{j,t,x}_{u}-X^{j,t,y}_{v}|^{2H_j-2}\Big|
 dudv\Big)\nonumber\\
&=:& \sum_{j=1}^d  A_{1j} (t,t,x,y,\tau_D)\,,
\label{e.def_Aj}
\end{eqnarray}
where we use the convention that $\prod_{i=k}^{k-1} a_j=1$
for any integer $k$ and any number $a_j$.
Now we use Lemma \ref{l.3.1} to obtain
for any $j=1, 2, \cdots, d$,
\bas
 A_{1j} (t,t,x,y,\tau_D)
&\le& C_d \mE^B\Big(\int_0^{t }\int_0^{t }|u-v|^{2H_0-2}\prod  _{i=1 }^{j-1}|B^{i,x }_u-B^{i,x }_v|^{2H_i-2}\prod  _{i=j+1 }^{d }|B^{i,x }_u-B^{i,y }_v|^{2H_i-2}\nonumber\\
&&\qquad\qquad \qquad
\Big||B^{j,x}_u-B^{j,x}_v|^{2H_j-2}-|B^{j, x}_{u}-B^{j, y}_{v}|^{2H_j-2}\Big|
 dudv\Big)\\
 &=& C_d    \int_0^{t }\int_0^{t }|u-v|^{\varrho-1}\prod  _{i=1 }^{j-1}\mE^B|\xi_i|^{2H_i-2}\prod  _{i=j+1 }^{d }\mE^B|\xi_i+x_i-y_i|^{2H_i-2}\nonumber\\
 &&\qquad\qquad \qquad
 \mE^B \Big||\xi |^{2H_j-2}-| \xi+z|^{2H_j-2}\Big|
 dudv    \,,
\eas
where $\xi, \xi_1, \cdots, \xi_d$ are independent standard Gaussian and $z=(x_j-y_j)/ \sqrt{|u-v|}$.
By the estimate (5.3) of \cite{Hu}, we see
\begin{equation}
A_{1j} (t,t,x,y,\tau_D) \le C|x_j-y_j|^{2\varrho}\,.
\label{e.6.3}
\end{equation}
[In fact, in \cite[Lemma A.6]{Hu} the computation is for $\mE \big(|\xi |^{2H_j-2}-| \xi+z|^{2H_j-2}\big) $.  But the same method works for $\mE \big||\xi |^{2H_j-2}-| \xi+z|^{2H_j-2}\big| $.]
 Thus, we have arrived at
\begin{equation}
	\mE^B |V_{t,\tau_D,x}-V_{t,\tau_D,y}|^2
	\le C |x-y|^{2\varrho} \,.
\end{equation}

{\bf Case (ii)}: When $s\neq t$ and $x=y$. Without loss of generality, we can assume that $s<t$.  We have
\ce
A(s,t,x,x,\tau_D)&=&\mE^B\Big(\int_0^{t\wedge\tau_D}\int_0^{t\wedge\tau_D}|u-v|^{2H_0-2}\prod^d_{i=1}|X^{i,t,x}_{u}-X^{i,t,x}_{v}|^{2H_i-2}dudv\\
&\qquad&+\int_0^{s\wedge\tau_D}\int_0^{s\wedge\tau_D}|u-v|^{2H_0-2}\prod^d_{i=1}|X^{i,s,x}_{u}-X^{i,s,x}_{v}|^{2H_i-2}dudv\\
&\qquad&-2\int_0^{t\wedge\tau_D}\int_0^{s\wedge\tau_D}|u-v|^{2H_0-2}\prod^d_{i=1}|X^{i,t,x}_{u}-X^{i,s,x}_{v}|^{2H_i-2}dudv\Big) \,.
\de
Denote the general integrand term
\[
\rho(u,v, t,s):=\prod^d_{i=1}|X^{i,t,x}_{u}-X^{i,s,x}_{v}|^{2H_i-2}\,.
\]
Then
\ce
&&A(s,t,x,x,\tau_D)\\
&=&\mE^B\Big(\int_0^{t\wedge\tau_D}\int_0^{t\wedge\tau_D}|u-v|^{2H_0-2}\rho(u,v, t,t) dudv +\int_0^{s\wedge\tau_D}\int_0^{s\wedge\tau_D}|u-v|^{2H_0-2}\rho(u,v, s,s)dudv\\
& &\qquad -2\int_0^{t\wedge\tau_D}\int_0^{s\wedge\tau_D}|u-v|^{2H_0-2}\rho(u,v, t,s)dudv\Big)\\
&=&\mE^B\Big(\int_0^{t\wedge\tau_D}\int_0^{t\wedge\tau_D}|u-v|^{2H_0-2}\rho(u,v, s,s) dudv +\int_0^{s\wedge\tau_D}\int_0^{s\wedge\tau_D}|u-v|^{2H_0-2}\rho(u,v, s,s)dudv\\
& &\qquad -2\int_0^{t\wedge\tau_D}\int_0^{s\wedge\tau_D}|u-v|^{2H_0-2}\rho(u,v, s,s)dudv\Big)\\
&&+\mE^B\Big(\int_0^{t\wedge\tau_D}\int_0^{t\wedge\tau_D}|u-v|^{2H_0-2}[\rho(u,v, t,t)-\rho(u,v, s,s)] dudv \\  &&\qquad\qquad -2\int_0^{t\wedge\tau_D}\int_0^{s\wedge\tau_D}|u-v|^{2H_0-2}[\rho(u,v,t,s) -\rho(u,v, s,s)]dudv\Big)\\
&=:&I_1+I_2\,.
 \de
The first expectation $I_1$ can be estimated by
\[
\begin{split}
I_1=&\mE^B\int_{s\wedge\tau_D}^{t\wedge\tau_D}\int_{s\wedge\tau_D}^{t\wedge\tau_D}|u-v|^{2H_0-2}\rho(u,v, s,s) dudv\,.
\end{split}
\]
It is clear that if $\tau_D\le s$, then $\tau_D\le t$, and the above integral is $0$.
This means
\[
\begin{split}
I_1=&\EE^B \int_{s }^{t\wedge\tau_D}\int_{s }^{t\wedge\tau_D}|u-v|^{2H_0-2}\rho(u,v, s,s) dudv
\\
\le &\EE^B\int_{s }^{t }\int_{s }^{t } |u-v|^{2H_0-2} \rho(u,v, s,s) dudv \\
\le &\int_{s }^{t }\int_{s }^{t }  |u-v|^{2H_0-2}\EE ^B \prod^d_{i=1}|B^{i  }_{u}-B^{i }_{v}|^{2H_i-2}
dudv \\
=&C _H  \int_{s }^{t }\int_{s }^{t }  |u-v|^{2H_0-2+\sum_{i=1}^dH_i-d}
dudv =C_H (t-s)^{\rho} \,.
\end{split}
\]
To bound $I_2$, we decompose it
to
\[
I_2= I_{21}+2I_{22}
\]
where
 \[
\begin{split}
 I_{21} =&\mE^B \left(\int_0^{t\wedge\tau_D}\int_0^{t\wedge\tau_D}|u-v|^{2H_0-2}[\rho(u,v, t,t)-\rho(u,v, s,s)] dudv\right)\\
 \le& \mE^B  \left(\int_0^{t }\int_0^{t }|u-v|^{2H_0-2}[\rho(u,v, t,t)-\rho(u,v, s,s)] dudv\right)\\
 \le&   C \left(\int_0^{t }\int_0^{t } |u-v|^{2H_0-2}  \mE^B
 \left|\rho(u,v, t,t)-\rho(u,v, s,s)\right|  dudv\right)
 \end{split}
 \]
 and
 \[
\begin{split}
I_{22}=&\EE^B \left(\int_0^{t\wedge\tau_D}\int_0^{s\wedge\tau_D}|u-v|^{2H_0-2}[\rho(u,v,s,s) -\rho(u,v, t,s)]dudv\right)\\
\le & \EE^B \left(\int_0^{t }\int_0^{s }|u-v|^{2H_0-2}[\rho(u,v,s,s) -\rho(u,v, t,s)]dudv\right)\\
\le&  C   \left(\int_0^{t }\int_0^{s }|u-v|^{2H_0-2}\EE ^B \left|\rho(u,v,s,s) -\rho(u,v, t,s)\right| dudv\right)\,.
\end{split}
\]
We shall deal with $I_{21}$ and $I_{22}$ can be dealt with similarly.    We write
\begin{equation}
\begin{split}
  \EE ^B \big| & \rho(u,v,t,t)    -\rho(u,v, s,s)\big| \\
  =& \EE ^B \left| \prod^d_{i=1}  |X^{i,t,x}_{u}-X^{i,t,x}_{v}|^{2H_i-2}-\prod^d_{i=1}  |X^{i,s,x}_{u}-X^{i,s,x}_{v}|^{2H_i-2}\right|  \\
 =& \sum_{j=1}^d \EE ^B \left| \prod_{i=1}^{j-1}    X^{i,t,x}_{u}-X^{i,t,x}_{v}|^{2H_i-2} \prod^d_{i=j+1}     |X^{i,s,x}_{u}-X^{i,s,x}_{v}|^{2H_i-2}
 \right|\\
 &\qquad\qquad
\left[  |X^{j,t,x}_{u}-X^{j,t,x}_{v}|^{2H_j-2}-   |X^{j,s,x}_{u}-X^{j,s,x}_{v}|^{2H_j-2}\right]  \\
=& \sum_{j=1}^d \left(\EE ^B \left| \prod_{i=1}^{j-1}    X^{i,t,x}_{u}-X^{i,t,x}_{v}|^{2H_i-2} \prod^d_{i=j+1}     |X^{i,s,x}_{u}-X^{i,s,x}_{v}|^{2H_i-2}
 \right|^p\right)^{1/p} \\
 &\qquad\qquad
\left(\EE^B \left|  |X^{j,t,x}_{u}-X^{j,t,x}_{v}|^{2H_j-2}-   |X^{j,s,x}_{u}-X^{j,s,x}_{v}|^{2H_j-2}\right|^q\right)^{1/q}   \,,
\end{split} \label{e.6.5}
\end{equation}
where $p$ and $q$ are two conjugate numbers to be determined later.  The first factor insider the sum  is dominated by
\[
\begin{split}
& \left(\EE ^B \left| \prod_{i=1}^{j-1}    X^{i,t,x}_{u}-X^{i,s,x}_{v}|^{2H_i-2} \prod^d_{i=j+1}     |X^{i,s,x}_{u}-X^{i,s,x}_{v}|^{2H_i-2}
 \right|^p\right)^{1/p} \\
&\qquad \qquad \le \left(\EE ^B \left| \prod_{i=1}^{j-1}    B_{u}-B_{v}|^{2H_i-2} \prod^d_{i=j+1}     |B_{u}-B_{v}|^{2H_i-2}
 \right|^p\right)^{1/p} \\
 &\qquad \qquad= C |u-v|^{\sum_{i\not= j} H_i-d+1}\,,
\end{split}
\]
if
\begin{equation}\label{sec6-cond.8}
( 2H_i-2 )p>-1
\quad\hbox{ for all $i=1, \cdots, d$}\,.
\end{equation}
The second factor  inside the expression \eqref{e.6.5}
can be estimated as follows.  First, we have
\begin{equation*}
\begin{split}
&  \left|  |X^{j,s,x}_{u}-X^{j,s,x}_{v}|^{2H_j-2}-   |X^{j,t,x}_{u}-X^{j,t,x}_{v}|^{2H_j-2}\right|  \\
&\qquad \le |X^{j,s,x}_{u}-X^{j,s,x}_{v}|^{2H_j-2}+   |X^{j,t,x}_{u}-X^{j,t,x}_{v}|^{2H_j-2}\\
&\qquad \le  2 \max (|X^{j,s,x}_{u}-X^{j,s,x}_{v}|^{2H_j-2},   |X^{j,t,x}_{u}-X^{j,t,x}_{v}|^{2H_j-2})\,.
\end{split}
\end{equation*}
Next,
\begin{equation*}
\begin{split}
&  \left|  |X^{j,s,x}_{u}-X^{j,s,x}_{v}|^{2H_j-2}-   |X^{j,t,x}_{u}-X^{j,t,x}_{v}|^{2H_j-2}\right|    \\
&\qquad\qquad =(2-2H_j)    \int_0^1 \left[ (1-\theta) |X^{j,t,x}_{u}-X^{j,t,x}_{v}|+\theta |X^{j,s,x}_{u}-X^{j,s,x}_{v})  |\right] ^{2H_j-3}\\
&\qquad\qquad\qquad \times
\hbox{sign} \left[ (1-\theta)   |X^{j,t,x}_{u}-X^{j,t,x}_{v}|+\theta |X^{j,s,x}_{u}-X^{j,s,x}_{v})  | \right]\\
&\qquad\qquad\qquad \times
\left[   |X^{j,s,x}_{u}-X^{j,s,x}_{v}|-|X^{j,t,x}_{u}-X^{j,t,x}_{v}|\right]   d\theta   \\
&\qquad\qquad \le   (2-2H_j) \int_0^1 \left[ (1-\theta) |X^{j,t,x}_{u}-X^{j,t,x}_{v}|+\theta |X^{j,s,x}_{u}-X^{j,s,x}_{v})  |\right] ^{2H_j-3}\\
&\qquad\qquad\qquad \times
\left|   |X^{j,s,x}_{u}-X^{j,s,x}_{v}|-|X^{j,t,x}_{u}-X^{j,t,x}_{v}|\right|   d\theta   \,.
\end{split}
\end{equation*}
Using $a+b\ge a^{1/2}b^{1/2}$ we see
\begin{equation*}
\begin{split}
&  \left|  |X^{j,s,x}_{u}-X^{j,s,x}_{v}|^{2H_j-2}-   |X^{j,t,x}_{u}-X^{j,t,x}_{v}|^{2H_j-2}\right|    \\
&\qquad\qquad \le    C  |X^{j,t,x}_{u}-X^{j,t,x}_{v}| ^{H_j-3/2}  |X^{j,s,x}_{u}-X^{j,s,x}_{v}   | ^{ H_j-3/2}\\
&\qquad\qquad\qquad \times
\left|   |X^{j,s,x}_{u}-X^{j,s,x}_{v}|-|X^{j,t,x}_{u}-X^{j,t,x}_{v}|\right|   \,.
\end{split}
\end{equation*}
Now we use the inequality that if for positive numbers $X\le A $ and $X\le B$, then for any $\al\in [0, 1]$,
$X\le A^\al B^{1-\al}$.  Thus, for any $\al\in [0, 1]$, we have
\begin{equation*}
\begin{split}
&  \left|  |X^{j,s,x}_{u}-X^{j,s,x}_{v}|^{2H_j-2}-   |X^{j,t,x}_{u}-X^{j,t,x}_{v}|^{2H_j-2}\right|    \\
&\qquad\qquad \le    C  |X^{j,t,x}_{u}-X^{j,t,x}_{v}| ^{(1-\al)(2H_j-2) +\al(H_j-3/2)}\\
&\qquad\qquad\qquad \times  |X^{j,s,x}_{u}-X^{j,s,x}_{v}   | ^{ (1-\al)(2H_j-2) +\al (H_j-3/2)}\\
&\qquad\qquad\qquad \times
\left|   |X^{j,s,x}_{u}-X^{j,s,x}_{v}|-|X^{j,t,x}_{u}-X^{j,t,x}_{v}|\right|^\al    \,.
\end{split}
\end{equation*}
Consequently,
\begin{equation*}
\begin{split}
& \EE  \left|  |X^{j,t,x}_{u}-X^{j,t,x}_{v}|^{2H_j-2}-   |X^{j,s,x}_{u}-X^{j,s,x}_{v}|^{2H_j-2}\right|^q    \\
&\qquad\qquad \le  \EE   \Big[  |X^{j,t,x}_{u}-X^{j,t,x}_{v}| ^{q(1-\al)(2H_j-2) +q\al(H_j-3/2)} \\
& \qquad \qquad \qquad  |X^{j,s,x}_{u}-X^{j,s,x}_{v}   | ^{ q(1-\al)(2H_j-2) +q\al (H_j-3/2)}\\
 &\qquad\qquad\qquad \times
\left|   |X^{j,s,x}_{u}-X^{j,s,x}_{v}|-|X^{j,t,x}_{u}-X^{j,t,x}_{v}|\right| ^{\al q }\Big]   \\
&\qquad\qquad \le  \Bigg(\EE   \Big[  |X^{j,t,x}_{u}-X^{j,t,x}_{v}| ^{q'q(1-\al)(2H_j-2) +q'q\al(H_j-3/2)}  \\
& \qquad \qquad \qquad  |X^{j,s,x}_{u}-X^{j,s,x}_{v}   | ^{ q'q(1-\al)(2H_j-2) +q'q\al (H_j-3/2)}\Big]\Bigg)^{1/q'} \\
 &\qquad\qquad\qquad \times
\left(\EE \left|   |X^{j,s,x}_{u}-X^{j,s,x}_{v}|-|X^{j,t,x}_{u}-X^{j,t,x}_{v}|\right| ^{p'\al q }   \right)^{1/p' } \,,
\end{split}
\end{equation*}
where $p'$ and $q'$ are two conjugate numbers to be specified later.  Using Cauchy-Schwarz inequality,
\begin{equation*}
\begin{split}
& \EE  \left|  |X^{j,t,x}_{u}-X^{j,t,x}_{v}|^{2H_j-2}-   |X^{j,s,x}_{u}-X^{j,s,x}_{v}|^{2H_j-2}\right|^q    \\
&\qquad\qquad \le  \Bigg(\EE   \Big[  |X^{j,t,x}_{u}-X^{j,t,x}_{v}| ^{q'q(1-\al)(4H_j-4) +q'q\al(2H_j-3 )}\Big]\Bigg)^{1/2q'}  \\
& \qquad \qquad \qquad  \Bigg(\EE   \Big[  |X^{j,s,x}_{u}-X^{j,s,x}_{v}   | ^{ q'q(1-\al)(4H_j-4) +q'q\al (2H_j-3 )}\Big]\Bigg)^{1/2q'} \\
 &\qquad\qquad\qquad \times
\left(\EE \left|   X^{j,s,x}_{u}-X^{j,s,x}_{v} -(X^{j,t,x}_{u}-X^{j,t,x}_{v})\right| ^{p'\al q }   \right)^{1/p' }  \\
&\qquad\qquad\qquad
=C |v-u|^{(2H_j-2)q} |t-s|^{q \al }   \,,
\end{split}
\end{equation*}
if
\begin{equation*}
q'q(1-\al)(4H_j-4) +q'q\al (2H_j-3 )>-1
\end{equation*}
Or
\begin{equation}\label{sec6-cond.9}
 q(1-\al)(4H_j-4) + q\al (2H_j-3 )>-1
 \end{equation}
since we can choose $q'$ arbitrarily close to $1$.
We need \eqref{sec6-cond.8} and \eqref{sec6-cond.9}
hold for all $j$  for all $i$ (and $j$).  This is possible
if
 \[
 (1-\al)(4H_j-4) +  \al (2H_j-3 )+2H_i-2>-1\,.
 \]
Or we need $H_j>5/6$ for any $j=1,2,\ldots,d$ and
we take
\begin{equation}\label{sec6-cond.alpha}
0<\al <\min_{1\le j\le d} \frac{6H_j-5}{2H_j-1}\,.
\end{equation}
Thus, under the condition  \eqref{sec6-cond.alpha},
\[
\EE |V_{t,\tau_D, x} -V_{s,\tau_D, x}|^2\leq|t-s|^{ \al } \,.
\]

 Combining this with \eqref{e.6.3}, we have
\begin{equation}
\mE|V_{t,\tau_D,x}-V_{s,\tau_D,y}|^2
\le C (|t-s|^{\varrho'}+|x-y|^{2\varrho})\,,
\end{equation}
where
\begin{equation}
\varrho'=\min(\varrho, \al )\,.
\end{equation}
With the establishment of this crucial
inequality, we can follows exactly the same argument
as in the Step 1 of the proof in \cite[Theorem 5.1]{Hu} to obtain
\ce
\mE^{W}|u(s,x)-u(t,y)|^p\leq C_p (
(t-s)^{ \varrho'  p/2} +|x-y|^{\varrho p}).
\de
Now a routine  application of the Kolmogorov continuity theorem
or the Garsia-Rodemich-Rumsey lemma yields the theorem.
\end{proof}
\begin{remark} It is easy to check that if $\al $ satisfies
\eqref{sec6-cond.alpha},  then  $\al\le 1$.
\end{remark}

\section{Sharp moment bounds  of the solution}
In this section we   obtain the matching upper and lower moment
bounds for the solution of parabolic Anderson model
\eqref{sec1-eq4}.  As indicated in \cite{Hu19} for the upper bound
there are several effective approaches.  However, for the lower moment
bound,  most researchers are able to obtain
asymptotics  only  for the second moment.

To obtain the sharp lower bounds for the general moment there are mainly two approaches.
For parabolic Anderson model \eqref{sec1-eq5},  one can use the powerful
  tool of Feynman-Kac formula   to obtain the sharp   lower $p$-moments
  bounds  of the solution and moreover, one can even obtain the exact asymptotics of the solution.  On the other hand for some other special stochastic partial
  differential equations when the Feynmna-Kac formula is not available, one   may combine   the chaos expansion and Feynman diagram formula to obtain the matching moment bounds. However, the exact asymptotics of the solution is still open
  (e.g. \cite{hw}).

  Since we have already established
the Feynman-Kac formula we expect to use this  formula to obtain the sharp moment (lower) bounds for the solution. This is the objective of this section.   As we know in applying the Feynman-Kac formula to obtain the matching moment bounds, a  critical tool is the small ball bound of the Brownian motion.  Thus,   we   need first to
establish  a small
ball estimate for general diffusion processes
\eqref{sec2-eq01}. To this end
we shall first assume that  the coefficients of the diffusion
in \eqref{sec2-eq01} is autonomous (time independent).
More specifically, in this section we make the following assumptions.
\begin{assumption}\label{a.7.1} We make the following assumptions:
\begin{enumerate}
\item[${\bf H^D}$:] Let $D$ be a   bounded $C^{1,1}$  domain   in $\mR^d$.

\item[${\bf H}^{b,\sigma}$:] The coefficients of the diffusion
satisfy the following conditions.\\

(i) $b: D \rightarrow \mR^d$ is locally bounded.\\

(ii) $\sigma: D \rightarrow \mR^{d}\otimes\mR^d$ is locally continuous and  is  uniformly elliptic, i.e.   there exists positive constant $\kappa$,
\ce
\kappa^{-1}|\xi|^2\leq \sum^d_{i,j=1}(\sigma\sigma')_{ij}(x) \xi_i\xi_j\leq \kappa |\xi|^2, \qquad  \forall \ x\in D\,,\ \ \xi\in\mR^d\,.
\de
\end{enumerate}
\end{assumption}

First, we need an estimate of small ball probability of the solution to equation (\ref{sec2-eq01}).   To this end
we consider   the  following boundary-value problem for the homogeneous second order differential operator:
\begin{equation}\label{sec6.2-eq1}
\begin{cases}
                  L\phi=\lambda(D)   \phi  &\qquad~in~\qquad~~D,\\
                  \phi=0~~ &\qquad~on~\qquad~~\partial D,
                 \end{cases}
\end{equation}
where
\begin{equation}\label{e.g.6.1}
L:=-\frac{1}{2}\sum^d_{i,j=1}(\sigma\sigma')_{ij}(x)\frac{\partial^2}{\partial{x_i}\partial{x_j}}-\sum^d_{i=1}b_i(x)\frac{\partial}{\partial{x_i}}.
\end{equation}
We use the notation $\lambda(D)$ to emphasize the dependence  of the eigenvalue on the domain $D$. In particular, we  will pay special
attention to  its dependence on $D$.
From the spectral theory of second order differentiable operators,
we know  the set of eigenvalues of $L$ is   countable.

The following theorem is   basic to us
 and is stated
for example, in \cite[Theorem 3 in   Chapter 6,
Section 5]{Evans}.
\begin{theorem}\label{Th6.2.1}
\begin{enumerate}
\item[(i)] There exists a real, positive eigenvalue $\lambda_1(D)$ for the operator $L$,  with the zero boundary condition. Furthermore, if $\lambda(D) \in \mC$ is any other eigenvalue, then
\ce
\Re(\lambda(D))\geq \lambda_1(D).
\de

\item[(ii)] The eigenvalue $\lambda_1(D)$ is simple, that is,  there is an
  eigenfunction $\psi_1$,  such that  any  other   eigenfunction  $\tilde \psi$  is a multiple of $\psi_1$:  $\tilde \psi=c\psi_1$ for some constant $c$.

\item[(iii)] The eigenfunction $\psi_1$ corresponding to principal eigenvalue $\lambda_1(D)$ can be chosen so that it is  positive on  $D$.
\end{enumerate}
\end{theorem}

The following three results can be found in \cite[Lemma 1.1, Theorem 2.1]{Nirenberg}.
\begin{theorem}\label{le6.2.2} Suppose ${\bf H}^D$ and ${\bf H}^{b,\sigma}$ hold true.
\begin{enumerate}
\item[(i)] Suppose $B_R=\{|x|<R\}$ lies inside  $D$, with $R\leq 1$. Then
\ce
\lambda_1(D)\leq \frac{C}{R^2}\,,
\de
where $C$ depends only on $d, \kappa$ and $\|b\|_{L^{\infty}(D)}$.
\item[(ii)] If we normalize $\psi_1$ so that $\psi_1(x_0)=1$ at a  fixed point $x_0\in D$, then $\psi_1\leq K$ in $D$ where $K$ is a constant depending only on $x_0, D, \kappa$ and    $\|b\|_{L^{\infty}(D)}$.
\item [(iii)] $\psi_1$ is bounded by a constant times the function $u_0$ where   $u_0\leq C \diam(D)|D|^{1/d}$ and $C$ depends only on $d, \kappa$ and $\|b\|_{L^{\infty}(D)}$.
\end{enumerate}
\end{theorem}

\begin{theorem}\label{Th6.2.3}
Let   $X_{\cdot}^{t,x}$ be  the Markov process associated with    equation (\ref{sec2-eq01}) whose   drift coefficients $b$ and diffusion coefficients $\sigma$ satisfy  ${\bf H}^D$ and ${\bf H}^{b,\sigma}$. Then there exist positive constants $c_1$ and $c_2$, independent $t, \vare, x$ such that
\begin{equation}\label{sec6.2-eq2}
\mP\left(\sup_{0<s\leq t}|X_s^{t,x}-x|<\varepsilon\right) \ge
 c_1\exp\left\{-\frac{ c_2t}{\varepsilon^2}\right\}\,.
\end{equation}
\end{theorem}
\begin{proof} We fix $x\in D$.
Let $D_\vare :=\{y\in \mR^d\ \big|\ |y-x|<\varepsilon\}\cap D$ (we omit the dependence of $D_\vare$ on $x$). Set $\tau_{D_\vare}:=\inf\{s>0, X^{t,x}_s\notin D_\vare \}$.   Then
$\mI_{t<\tau_{D_\vare}}$ is a multiplicative functional (e.g.
\cite[III.3.7]{bg} and \cite[III.18]{rw}),  and $u_\vare(t,x)=\mE^x\left(f(X_t^{t,x})\mI_{t<\tau_{D_\vare} }\right)$ is the solution to the following  initial-boundary value problem
\begin{equation}\label{sec6.2-eq3}
\begin{cases}
                  \partial_t u_\vare =Lu_\vare,       & \qquad x\in  D_\vare,\\
                  u_\vare (0,x)=f(x), & \qquad x\in  D_\vare  , \\
                  u_\vare (t,x)=0,  & \qquad   x\in \partial D_\vare,\qquad t>0.
                 \end{cases}
\end{equation}
[See also \cite{hu92} and \cite[Theorem 5.13, or Example 5.19]{hubook} for Brownian motion case.]
By the spectral theory \cite{Evans} we have
\begin{equation}
u_\vare (t,x)=\sum^{\infty}_{n=1}e^{-\lambda_n(D_\vare)t}\phi_{\vare, n}(x)\int_{D_\vare} \phi_{\vare, n}(y)f(y)dy,
\label{e.spectral}
\end{equation}
where $\lambda_1( D_\vare), \lambda_2(D_\vare), \lambda_3(D_\vare), \ldots$ are eigenvalues of $L$, $\lambda_1(D_\vare)>0$ and $\Re(\lambda_n(D_\vare))> \lambda_1(D_\vare)$, $n=2,3,\ldots$,  $\{\phi_{\vare, n} (x)\}_{n\geq1}$ are corresponding eigenfunctions of the eigenvalue problem
\begin{equation}\label{sec6.2-eq4}
\left\{
                  \begin{array}{lll}
                  L\phi_{\vare, n}-\lambda_n(D_\vare)\phi_{\vare, n}=0  ~\qquad~~in~\qquad~~D_\vare,\\
                  \phi_{\vare, n}(x)=0, \qquad x\in \partial D_\vare.\\
                  \end{array}
                    \right.
\end{equation}
And $\phi_{\vare, n}$ forms an orthonormal basis of $L^2(D_\vare, dx)$,  in particular, $ \int_{ D_\vare} \phi_{\vare, n}^2 (x)dx=1$  for all $n\ge 1$.

Taking $f=1$   in \eqref{e.spectral}  and    by
Theorem \ref{Th6.2.1}, Proposition   \ref{le6.2.2}, we have
\ce
\mP\left(\sup_{0<s\leq t}|X_s^{t,x}-x|<\varepsilon\right)
&=&\mP(\tau_{D_\vare}>t)\\
&=&\sum^{\infty}_{n=1}e^{-\lambda_n(D_\vare)t}\phi_{\vare, n}(0)\int_
{D_\vare} \phi_{\vare, n}(y)dy\\
&\sim& e^{-\lambda_1({D_\vare})t}\phi_{\vare, 1}(0)\int_{D_\vare}\phi_{\vare, 1}(y)dy\\
&\geq&\exp\left\{-\frac{C}{\varepsilon^2}t\right\}\phi_{\vare, 1}(0)\int_{D_\vare}\phi_{\vare, 1}(y)dy.\\
\de
By Proposition \ref{le6.2.2} (iii), we know that $\phi_{\vare, 1}(x)$ is bounded by a positive constant $K$ which is independent of $\vare\in (0, 1)$
and is with $L^2(\D_\vare, dx)$-norm $1$. Thus,  we have
\[
1=\int_{ D_\vare} \phi_{\vare, 1}^2 (x)dx \le \int_{ D_\vare}
\phi_{\vare, 1}  (x) \left[\sup_{x\in D_\vare} \phi_{\vare, 1}  (x)\right] dx \le
K \int_{ D_\vare} \phi_{\vare, 1}  (x)dx\,.
\]
This implies
\[
\int_{D_\vare}\phi_{\vare, 1}(y)dy\ge 1/K\,.
\]
Hence, we have
\[
\mP\left(\sup_{0<s\leq t}|X_s^{t,x}-x|<\varepsilon\right)
 \geq c_1\exp\left\{-\frac{C_2}{\varepsilon^2}t\right\} \\
\]
This proves the theorem.
\end{proof}

Now we can use this small ball    estimate to obtain the following sharp $p$-th moment bounds of the solution to equation (\ref{sec1-eq4}). We assume that the coefficients $\sigma, b$ are time independent when we estimate the   lower bound, while we can allow  $\sigma, b$ to be  time dependent when we estimate the  upper bound.

\begin{theorem}\label{Th6.2.5} Suppose $2H_0+\sum^d_{i=1}H_i-d-1>0$.
	Assume that $u$ is the solution to the Stratonovich equation (\ref{sec1-eq4}).
	\begin{enumerate}
	\item[(i)] For any $p\geq 2$, we have
\begin{eqnarray}\label{sec6.2-eq6}
\mE(u^p(t,x))
&\leq& C_1\exp\left\{C_2t^{\frac{2H_0+\sum^d_{i=1}H_i-d}{\sum^d_{i=1}H_i+1-d}}p^{\frac{\sum^d_{i=1}H_i+2-d}{\sum^d_{i=1}H_i+1-d}}\right\}
\end{eqnarray}
for all $t\geq 0$, $x\in D$,  where $C_1,C_2 $ are constants independent of $t, x$ and $p$.
\item[(ii)]  If furthermore Assumption \ref{a.7.1} (namely,   ${\bf H}^D$, ${\bf H}^{\sigma,b}$)   also holds,  and assume that $h(t,x)\ge c>0$,  then for any $p\geq 2$, we have
\begin{eqnarray}\label{sec6.2-eq6a}
\mE(u^p(t,x))\ge 	C_3\exp \left\{C_4 t^{\frac{2H_0+\sum^d_{i=1}H_i-d}{\sum^d_{i=1}H_i+1-d}}p^{\frac{\sum^d_{i=1}H_i+2-d}{\sum^d_{i=1}H_i+1-d}}\right\}
\end{eqnarray}
for all $t\geq 0$, $x\in D$,  where $C_3,C_4 $ are constants independent of $t, x$ and $p$.
\end{enumerate}
\end{theorem}
\begin{proof}
It suffices to prove the case when  $p=k=2,3,\ldots$
is an integer. It is a routine argument from  all positive  integers to all positive real numbers via H\"older inequality.
We denote by $X^{ l,t,x}$, $l=1, 2, \ldots ,k$ the $k$ independent copy of the   $d$-dimensional solution   of equation (\ref{sec1-eq4}) and we denote by $X^{m, l,t,x}$, $m=1, 2, \ldots ,d$ the $m$-th component   $X^{ l,t,x}$. Using the Feynmann-Kac formula (\ref{sec5-eq2}) for the solution to equation (\ref{sec1-eq4}), and applying Cauchy-Schwartz inequality yield
\ce
\mE(u^k(t,x))
&=&\mE^B\Big[\prod^k_{l=1}h(t\wedge\tau_D, X^{l,t,x}_{t\wedge\tau_D})\exp\Big(\frac{\alpha_H}{2}\sum_{1\leq i, j\leq k}\int_0^{t\wedge\tau_D}\int_0^{t\wedge\tau_D}|s-r|^{2H_0-2}\\
&&\qquad \times \prod^d_{m=1}|X_s^{m,i,t,x}-X_r^{m,j,t,x}|^{2H_m-2}dsdr\Big)\Big]\\
&\leq& C^k  \Big(  \mE^B\Big[ \exp\Big(\frac{q\alpha_H}{2}\sum_{1\leq i, j\leq k}\int_0^{t }\int_0^{t }|s-r|^{2H_0-2}\\
&&\qquad \times \prod^d_{m=1}|X_s^{m,i,t,x}-X_r^{m,j,t,x}|^{2H_m-2}dsdr\Big)\Big]\Big)^{1/q} \\
&\leq& C^k  \Big(  \mE^B\Big[ \exp\Big(C \sum_{1\leq i, j\leq k}\int_0^{t }\int_0^{t }|s-r|^{2H_0-2}\\
&&\qquad \times \prod^d_{m=1}|B_{\kappa_2s}^{m,i, x}-B_{\kappa_2 r}^{m,j, x}|^{2H_m-2}dsdr\Big)\Big]\Big)^{1/q} \\
&\leq& C^k  \Big(  \mE^B\Big[ \exp\Big(C \sum_{1\leq i, j\leq k}\int_0^{t/\kappa_2   }\int_0^{t/\kappa_2  }|s-r|^{2H_0-2}\\
&&\qquad \times \prod^d_{m=1}|B_{ s}^{m,i, x}-B_{  r}^{m,j, x}|^{2H_m-2}dsdr\Big)\Big]\Big)^{1/q}\,,
\de
where the above second inequality follows from Lemma \ref{l.3.3}.
Comparing  the above expression with the expression of $\EE [u_{t,x}^k]$ in the proof of
\cite[Theorem 6.4]{hhnt} (see also \cite{bc}) we can use the second inequality of \cite[Inequality (6.1)]{hhnt} to obtain
\[
\EE \left[ u^k(t,x)\right]\le
C^k \exp \left( C' (t/\kappa_2)^{\frac{4-2\beta -a}{2-a}} k^{\frac{4   -a}{2-a}}\right)\,.
\]
where $\beta$ is given in \cite[Hypothesis 6.1]{hhnt}
and $a$ is given in \cite[Theorem 6.4]{hhnt}.
In our case we have
\[
\ga (t)=|t|^{2H_0-2}\quad \hbox{and}\quad
\La(x)=\prod_{i=1}^d |x_i|^{2H_i-2}\,.
\]
Therefore,
\[
\beta=2-2H_0\,,\quad a=\sum_{i=1}^d (2-2H_i)=
2d-2|H|\,,
\]
where
\[
|H|=\sum_{i=1}^d H_i\,.
\]
Therefore we obtain the upper bound
\begin{equation}\label{sec6-eq7}
\mE(u^k(t,x)) \leq C^k\exp\Big(C t^{\frac{ 2H_0  + |H|- d}{1+ |H|  - d}}
	k^{\frac{2  + |H| - d}{1+ |H|  - d}} \Big).
\end{equation}
This proves the inequality \eqref{sec6.2-eq6} when $p$ is an integer $k$.

As for the lower bound, we shall take into account again the Feynman-Kac formula (\ref{sec5-eq2}) for the moments of the solution $u$ and the small ball probability estimates for the solution of (\ref{sec3-eq4}).  We denote by $X^{m,l,t,x}$, $m=1, 2, \ldots ,d$ the $m$-th component of the $d$-dimensional $X^{l,t,x}$ and let
$$
A_{\varepsilon,t}:=
\left\{\sup_{1\leq i<j\leq k}\sup_{1\leq m\leq d}\sup_{0\leq s,r\leq t\wedge \tau_D}|X^{m,i,t,x}_s-X^{m,j,t,x}_r|\leq \varepsilon \right\}.
$$
Then, owing to formula (\ref{sec5-eq2}) it is easy to see that
\ce
\mE(u^k(t,x))&=&\mE^B\Big[\prod^k_{l=1}h(t\wedge\tau_D, X^{l,t,x}_{t\wedge\tau_D})\exp\Big(\frac{\alpha_H}{2}\sum_{1\leq i, j\leq k}\int_0^{t\wedge\tau_D}\int_0^{t\wedge\tau_D}|s-r|^{2H_0-2}\\
&&\qquad \times \prod^d_{m=1}|X_s^{m,i,t,x}-X_r^{m,j,t,x}|^{2H_m-2}dsdr\Big)\Big]\\
&\geq& \mE^B\Big[C^k\exp\Big(\frac{\alpha_H}{2} \sum_{1\leq i< j\leq k}\int_0^{t\wedge\tau_D}\int_0^{t\wedge\tau_D}|s-r|^{2H_0-2}\\
&&\qquad \times \prod^d_{m=1}|X_s^{m,i,t,x}-X_r^{m,j,t,x}|^{2H_m-2}dsdr\Big)\mI_{\{A_{\varepsilon,t}\}}\mI_{\{t<\tau_D\}}\Big]\\
&\geq& \mE^B\Big[C^k\exp\Big(\frac{\alpha_H}{2} \sum_{1\leq i< j\leq k}\int_0^t\int_0^t|s-r|^{2H_0-2}\\
&&\qquad \times \varepsilon^{2\sum^d_{m=1}H_m-2d}dsdr\Big)\mI_{\{A_{\varepsilon,t}\}}\mI_{\{t<\tau_D\}}\Big]\\
&=&C^k\exp\left\{\frac{k(k-1)\alpha_H}{2H_0(2H_0-1)}t^{2H_0}\varepsilon^{2|H| -2d}\right\}\mP(A_{\varepsilon,t}\cap \{t<\tau_D\})\\
&\geq&C^k\exp\left\{C_{H_i}k^2t^{2H_0}
\varepsilon^{2 |H| -2d}\right\}\mP(A_{\varepsilon,t}\cap \{t<\tau_D\}).
\de
Denote
\[
 F_{i ,t}:=\{\sup_{0\leq s\leq t\wedge \tau_D}|X^{i, t,x}_s-x|\leq \varepsilon/2\}.
\]
It is easy to see that
\[
 F_{i ,t}
 \subseteq  \left\{ \sup_{1\leq m\leq d}\sup_{0\leq s\leq t\wedge \tau_D}|X^{m,i,t,x}_s-x |\leq \varepsilon /2 \right\}\,.
 \]
Moreover,  we have
\ce
  \cap^k_{l=1}F_{ l,t}\subset A_{\varepsilon,t}\,.
\de
The events $F_{ l,t}, l=1,2, \cdots,k$ being i.i.d,  we have
\ce
\mP(A_{\varepsilon,t}\cap \{t<\tau_D\})&\geq&
\left[ \mP (F_{ l,t}\cap \{t<\tau_D\})\right]^k \\
&=&\left[ \mP (\{\sup_{0\leq s\leq t}|X^{ l,t,x}_s-x|\leq \varepsilon/2\}\cap \{t<\tau_D\})\right]^k \,.
\de
We can choose $\vare$ sufficiently small so that
\[
\left\{ y\in D\,; \ |y-x|\leq\vare/2\right\}
\subseteq D\,.
\]
 Theorem \ref{Th6.2.3}  yields
\ce
\mP(A_{\varepsilon,t}\cap \{t<\tau_D\})&\geq&
\left[ \mP (F_{ l,t}\cap \{t<\tau_D\})\right]^k \\
&=&\left[ \mP (\{\sup_{0\leq s\leq t}|X^{ l,t,x}_s-x|\leq \varepsilon/2\} )\right]^k \\
&\ge&  c_1\exp\left\{-\frac{k C_2}{\varepsilon^2}t\right\} \,.
\de

Therefore we have
\ce
\mE(u^k(t,x))\geq C^kc_3\exp\left\{C_{H_i}k^2t^{2H_0}\varepsilon^{2|H| -2d}-\frac{c_4tkd}{\varepsilon^2}\right\}.
\de
The above inequality holds true for any $\vare$
sufficiently small. We can
  optimize the expression on right side to obtain the best bounds by this approach.
  In fact we can let
  \[
  C_{H_i}k^2t^{2H_0}\varepsilon^{2|H| -2d}=2\frac{c_4tkd}{\varepsilon^2}.
  \]
Or we let  \[
  \varepsilon^2=\left(\frac{2c_4d}{C_{H_i}kt^{2H_0-1}}\right)^{\frac{1}{|H| +1-d}}
  \]
  which can be arbitrarily small when $k$ is sufficiently large.  This  yields
\begin{equation}\label{sec6-eq8}
\mE(u^k(t,x))\geq C^kc_3\exp\left\{C_{H_i}t^{\frac{2H_0+\sum^d_{i=1}H_i-d}{\sum^d_{i=1}H_i+1-d}}k^{\frac{\sum^d_{i=1}H_i+2-d}{\sum^d_{i=1}H_i+1-d}}\right\}.
\end{equation}
This proves \eqref{sec6.2-eq6a} when $p$ is a
sufficiently large  integer $k$.

To go from any positive numbers  $p\geq 2$, we can use the following inequality
\ce
\||u(t,x)\|_{p-1}\leq \|u(t,x)\|_{[p]}\leq \|u(t,x)\|_{p}
\de
where $[\cdot]$ means the integer part of real number.
This accomplishes the proof of this theorem.
\end{proof}

A straightforward consequence of this theorem is
\begin{corollary} Suppose $2H_0+\sum^d_{i=1}H_i-d-1>0$. Assume that $u$ is the solution to the Stratonovich equation (\ref{sec1-eq4})  and  Assumption \ref{a.7.1} holds  true and assume that $h(t,x)\ge c>0$. Then for any $p\geq2$, we have
\ce
&&C^p\exp\left\{C\alpha_Ht^{\frac{2H_0+\sum^d_{i=1}H_i-d}{\sum^d_{i=1}H_i+1-d}}p^{\frac{\sum^d_{i=1}H_i+2-d}{\sum^d_{i=1}H_i+1-d}}\right\} \leq  \mE(u^p(t,x))\\
 &&\qquad\qquad \qquad \leq  C^pc_3\exp\left\{C_{H_i}t^{\frac{2H_0+\sum^d_{i=1}H_i-d}{\sum^d_{i=1}H_i+1-d}}p^{\frac{\sum^d_{i=1}H_i+2-d}{\sum^d_{i=1}H_i+1-d}}\right\}.
\de
	\end{corollary}

\end{document}